\documentclass[10pt]{amsart}
\usepackage{amssymb}
\usepackage{amsmath}
\usepackage{latexsym,enumerate}
\usepackage{enumerate,graphicx}
\usepackage{array,subfig}
\usepackage{xcolor}
\usepackage[ulem=normalem]{changes}
\usepackage{color}
\definecolor{purple}{rgb}{0.62,0.12,0.94}
\definechangesauthor{ERJ}{blue}
\newtheorem{theorem}{Theorem}[section]
\newtheorem{lemma}[theorem]{Lemma}
\newtheorem{proposition}[theorem]{Proposition}
\newtheorem{corollary}[theorem]{Corollary}
\theoremstyle{definition}
\newtheorem{definition}{Definition}[section]
\theoremstyle{remark}
\newtheorem{remark}{Remark}[section]
\theoremstyle{remark}

\newcommand{\supp}{\operatorname{supp}}

\newcommand{\eps}{\ensuremath{\varepsilon}}
\newcommand{\R}{\ensuremath{\mathbb{R}}}

\newcommand{\del}{\partial}
\newcommand{\ol}{\overline}
\newcommand{\ul}{\underline}
\newcommand{\ra}{\rightarrow}
\newcommand{\alp}{\alpha}

\newcommand{\U}{\hspace{1pt}\mathcal U}

\newcommand{\bx}{\bar{x}}
\newcommand{\by}{\bar{y}}
\newcommand{\bz}{\bar{z}}

\newcommand{\sgn}{\operatorname{sgn}}

\def\dom{\Omega}
\def\bound{\partial\Omega}
\def\domb{\overline{\Omega}}
\def\tibeta{{\overline \beta}}

\numberwithin{equation}{section}
\newlength{\figwidth}
\title[Neumann problems]{On Neumann type problems for non-local equations set
  in a half space}

\author[Barles]{Guy Barles}
\address[Guy Barles]{
Laboratoire de Math\'ematiques et Physique Th\'eorique, CNRS UMR 6083, F\'ed\'eration Denis Poisson,
Universit\'e Fran\c{c}ois Rabelais, Parc de Grandmont,
37200 Tours, France}
\email{barles@lmpt.univ-tours.fr}

\author[Chasseigne]{Emmanuel Chasseigne}
\address[Emmanuel Chasseigne]{
Laboratoire de Math\'ematiques et Physique Th\'eorique, CNRS UMR 6083, F\'ed\'eration Denis Poisson,
Universit\'e Fran\c{c}ois Rabelais, Parc de Grandmont,
37200 Tours, France}
\email{Emmanuel.Chasseigne@lmpt.univ-tours.fr}

\author[Georgelin]{Christine Georgelin}
\address[Christine Georgelin]{
Laboratoire de Math\'ematiques et Physique Th\'eorique, CNRS UMR 6083, F\'ed\'eration Denis Poisson,
Universit\'e Fran\c{c}ois Rabelais, Parc de Grandmont,
37200 Tours, France}
\email{christine.georgelin@lmpt.univ-tours.fr}

\author[Jakobsen]{Espen R.~Jakobsen}
\address[Espen R. Jakobsen]{
    Department of Mathematical Sciences,
    Norwegian University of Science and Technology,
    7491 Trondheim, Norway}
\email{erj\@@math.ntnu.no}

\thanks{Jakobsen was supported by the Research
Council of Norway through the project ``Integro-PDEs: Numerical
methods, Analysis, and Applications to Finance''.}



\begin{document}
\begin{abstract}
  We study Neumann type boundary value problems for nonlocal equations
  related to L\'evy processes. Since these equations are nonlocal,
  Neumann type problems can be obtained in many ways, depending on
  the kind of ``reflection'' we impose on the outside
  jumps. To focus on
  the new phenomenas and ideas, we consider different models
  of reflection and rather general non-symmetric L\'evy
  measures, but only simple linear equations in
  half-space domains. We derive the Neumann/reflection problems
  through a truncation procedure on the L\'evy measure, and then we
  develop a viscosity solution theory which includes
  comparison, existence, and some regularity results. For problems
  involving fractional Laplacian type operators like e.g.
  $(-\Delta)^{\alp/2}$, we prove that
  solutions of all our nonlocal Neumann problems converge as $\alp\ra
  2^-$ to the solution of a classical Neumann problem.  The
reflection  models
we consider include cases where the underlying
  L\'evy processes are reflected, projected, and/or censored
  upon exiting the domain.
\end{abstract}
\maketitle


\section{Introduction}
\label{Sec:intro}

In the classical probabilistic approach to elliptic and parabolic
partial differential equations via Feynman-Kac formulas, it is well-known
that Neumann type boundary conditions are associated to stochastic
processes having a reflection on the boundary. We refer the reader to
the book of Freidlin~\cite{MF:Book} for an introduction and to Lions
and Sznitman \cite{LS} for general results. A key result in this
direction is roughly speaking the following: for a PDE with Neumann or
oblique boundary conditions, there is a unique underlying reflection
process, and any consistent approximation will converge to it in
the limit (see \cite{LS} and Barles \& Lions \cite{BL}). At least in the case of
normal reflections, this result is strongly connected to the study
of the Skorohod problem and relies on the
underlying stochastic processes being continuous.

The starting point of this article was to address the same question, but now
for jump diffusion processes related to partial
integro-differential equations (PIDEs in short). What is a reflection
for such processes, and is a PIDE with Neumann boundary conditions
naturally connected to a reflection process? It turns out that the
situation is more complicated in this setting,  at least the
questions have to be reformulated in a slightly different way. In this
article we address these questions through an
analytical PIDE approach where  we keep in mind the idea
of having a reflection process but without defining it precisely or
even proving its existence.

For jump processes which are discontinuous and may exit a domain
without touching its
boundary, it turns out that there are many ways to define a
``reflection'' or a ``process with a reflection''. This remains true
even if we restrict ourselves to a mechanism
which is connected to a Neumann boundary condition (see below). But
because of the way the PIDE and the process are related, defining a
reflection on the boundary will change the equation inside the
domain. This is a new nonlocal phenomenon which is not encountered in the case
of continuous processes and PDEs.
\\

\noindent\textsc{PIDE with Neumann-type boundary condition --}
In order to simplify the presentation of paper and focus on the main new ideas
and phenomenas, we consider different models
  of reflections and rather general non-symmetric L\'evy
  measures, but only for problems
involving linear equations set in simple domains. The cases we will consider
already have interesting features and difficulties. To be precise, we
consider half space domains $\dom:=\big\{(x_1,\dots,x_N)=(x',x_N)\in\R^N:
x_N> 0\big\}$ and simple linear Neumann type problems that
we write as
\begin{align}
\label{E}
\begin{cases}
u(x)-I[u](x)-f(x)=0 &\text{in}\quad \dom,\\
-\frac{\del u}{\del x_N}=0&\text{in}\quad \del\dom,
\end{cases}
\end{align}
or sometimes as
$$
\begin{cases}
F(x,u,I[u])=0 &\text{in}\quad \dom,\\
-\frac{\del u}{\del x_N}=0&\text{in}\quad \del\dom,
\end{cases}
$$
where $F(x,r,l)=r-l -f(x)$ and
\begin{align*}
I[u](x)&=
\lim_{b\ra0^+}\int_{b<|z|} \left [u(x+\eta(x,z))-u(x)\right ]\ d\mu(z).
\end{align*}
We will assume that $f\in C_b(\domb)$, i.e. $f$ is bounded and continuous,
that $\mu$ is a nonnegative Radon measure satisfying
\begin{align}
\label{Levy}
 \int |z|^2\wedge 1\ d\mu(z)<\infty,
\end{align}
and that
\begin{align}
\label{n-cond}
 x+\eta(x,z)\in\domb \text{ for all }
 x\in\domb\,,\ \eta(x,z)=z\text{ if }x+z\in\domb.
\end{align}

Note that $I[u]$ is a principal value ($\mathrm{P.V.}$) integral, and
that \eqref{Levy} is the most general integrability assumption
satisfied by the L\'evy measure associated to any L\'evy
process \cite{A:Book}.
When
$\eta(x,z)\equiv z$, then $I[u]$ is the generator of a stochastic process
which can jump from $x \in \domb$ to $x+z$ with a certain probability,
see e.g. \cite{A:Book,CT:Book,FOT}.
Assumption \eqref{n-cond} is a type of reflection condition
preventing the jump-process from leaving the domain: nothing happens and
$\eta(x,z)=z$ if $x+z\in\domb$, while if $x+z \notin \domb$, then a
"reflection'' is performed in order to move the particle back to a
point $P(x,z)= x+\eta(x,z)$ inside $\domb$.
Note that  we have to check at some point
that the reflection is consistent with a Neumann boundary
condition.

The main examples of $\eta$ are the following model cases, where we
use the notation $x=(x',x_N)\in\R^{N-1}\times\R_+$,
$\eta(x,z)=(\eta(x,z)',\eta(x,z)_N)$, etc.:
\begin{align*}
&\mathrm{(a)}\quad  \eta(x,z)=\begin{cases}z \hspace{0.6cm}\  & \text{if }
 x_N+z_N\geq0\\
0 & \text{if not}
 \end{cases}&& \text{[censored]}\\
&\mathrm{(b)}\quad  \eta(x,z)=\begin{cases}z & \text{if }
 x_N+z_N\geq0\\
z\frac{x_N}{|z_N|} & \text{if not}
 \end{cases}  &&\text{[fleas on the window]}\\
&\mathrm{(c)}\quad  \eta(x,z)=\begin{cases}z & \text{if }
 x_N+z_N\geq0\\
(z',-x_N) & \text{if not}
 \end{cases} &&\text{[normal projection]}\\
&\mathrm{(d)}\quad  \eta(x,z)=\begin{cases}z & \text{if }
 x_N+z_N\geq0\\
(z',-2x_N-z_N) & \text{if not}
 \end{cases} &&\text{[mirror reflection]}
\end{align*}
for all $x\in\dom$ and $z\neq0$. The different reflections are
depicted in the figure below.
\begin{center}
    \includegraphics[scale=0.72]{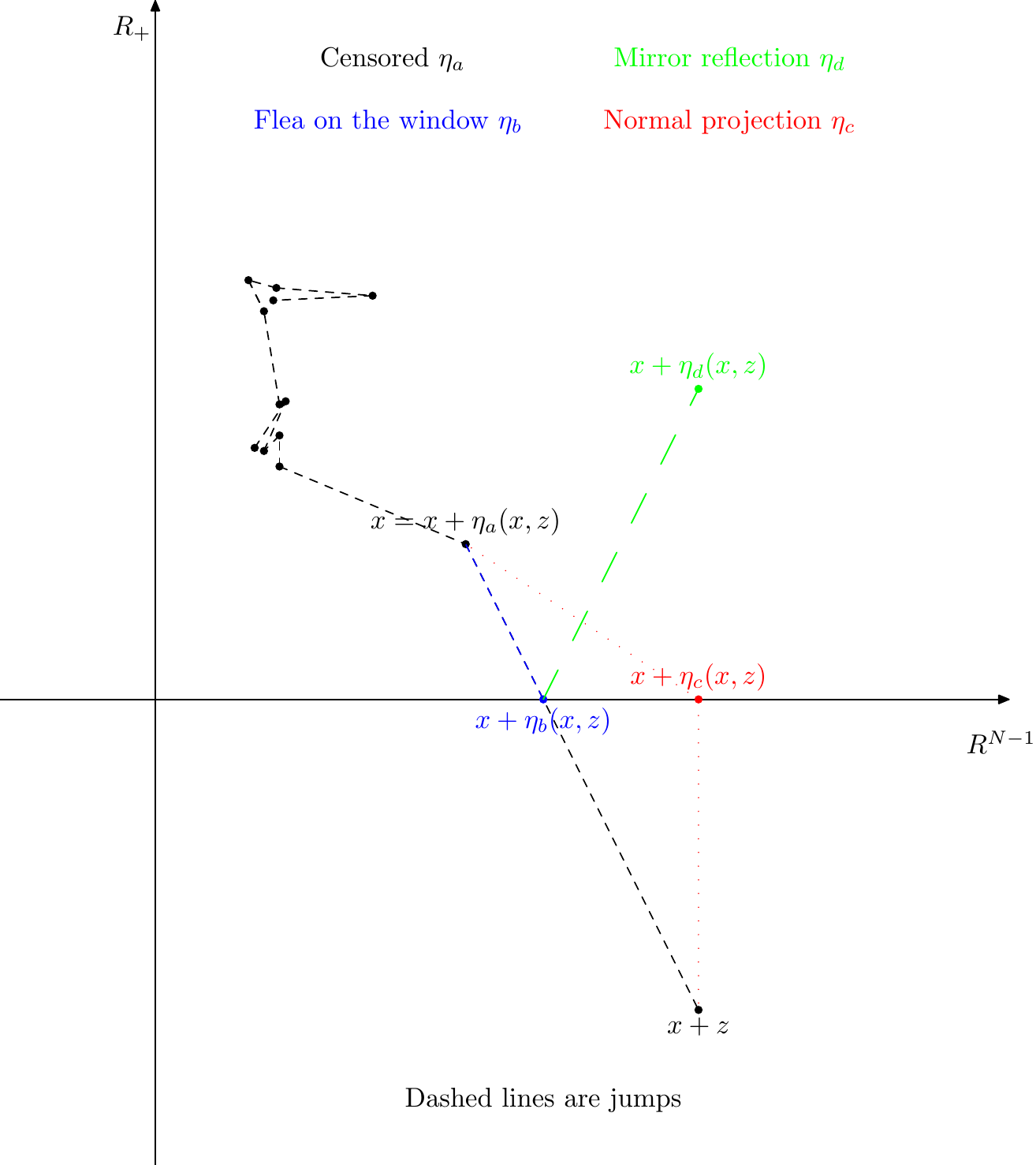}
\end{center}
We will discuss later whether the naively proposed "reflections" are
realized by a concrete Markov process, \textit{i.e.} if they correspond to the generator of such a process.
\\

\noindent\textsc{Main results --}
From an analytical (PIDE) point of view, we first have to give a sense
to problem \eqref{E} and relate it to an homogeneous  Neumann boundary
value problem. This is done in Sections~\ref{ass-def} and
\ref{der-pide}. The first part is classical: to take into
account singular L\'evy measures, we split the integral operator in
two and write the equation in a more convenient way. Here
classical arguments in viscosity solution
theory are used, see e.g. \cite{BI} and references therein. Viscosity
solution theory is also used to give a good definition of the Neumann
boundary conditions.

If $\mu$ is a nice bounded measure, then the problem
(\ref{E}) can be solved easily without caring much about
the Neumann boundary condition. Moreover, the solutions will be uniformly
bounded by $||f||_\infty$. Intuitively (\ref{E}) carries the
information that the particles remain in $\domb$ since they can
only jump inside $\domb$. This mass conservation is an other
way to understand that we are dealing with a (homogeneous) Neumann
type of boundary condition.

When $\mu$ then is a singular measure, we can approximate it by a
sequence of bounded measures
 $(\mu_n)_n$, consider the associated (uniformly bounded!) solutions
 $(u_n)_n$, and wonder
 what the limiting problem is if $\mu_n$ converges to $\mu$ in a
 suitable sense. This is the way we choose to make sense of both the
 definition of problem (\ref{E}) and the associated notion of (viscosity)
solutions. We point out here that the ``real'' Neumann boundary
condition arises only if the measure is singular enough. In the other
cases, either the process will never reach the boundary as in the
censored case for $\alpha$-stable process with $\alpha<1$ (see e.g. \cite
{BBC}), or the equation will hold up to the boundary.

The natural next step is then to prove uniqueness results for all the
above models and equations. In this paper different types of proofs
are given depending on the singularity of the measure and
the structure of the ``reflection'' mechanism at the boundary. These
results are given in Sections \ref{sec:nonsens} -- \ref{sec:sens2}.
The first case we treat is when the jump function $\eta$ enjoys a contraction
property in the normal direction. This covers all the non-censored cases
listed above -- models (b)--(d). Had we had a contraction in all
directions, then the usual viscosity solution doubling of variables
argument would work. Here we have to modify that argument to take into
account the special role of the normal direction.

The second case we consider is the censored case (a) when the
singularity of the measure is not too strong. By this we mean
typically a stable process with L\'evy measure with density
$\frac{d\mu}{dz}\sim\frac{1}{|z|^{N+\alpha}}$ for $\alpha\in(0,1)$. We
construct an approximate subsolution which blows up at the
boundary and this allows us to derive a comparison result by a
penalization procedure. Such a construction is related to the fact
that the process does not reach the boundary in this case, see
e.g. \cite{BBC}.

The last case is the censored case (a) when the singularity is
strong, \textit{i.e.} when $\alpha\in[1,2)$. This case requires much more
work because no blow-up subsolutions exist here. In fact, when
$\alpha\in[1,2)$, the censored process does reach the boundary (see
e.g. \cite{BBC}). We first prove that the Neumann boundary
condition is already encoded in the equation under the additional
assumption that the solution is $\beta$-H\"older continuous at the
boundary for some $\beta>\alpha-1$. Then we prove a comparison result
for sub/super solutions with this H\"older regularity at
the boundary. The proof is then similar to the proof in the $\alp<1$ case,
except that the special subsolutions in this case are bounded and only
penalize the boundary when the sub/super solutions are H\"older
continuous there. Finally, we construct solutions in this
class. In dimension $N=1$, we use and prove that any bounded
uniformly continuous solution is H\"older continuous provided $\mu$
satisfies some additional integrability condition. In higher
dimensions, we use and prove a similar result under additional regularity
assumptions on $f$ in the tangential directions.

Finally, in Section \ref{sec:limit} we show that  all the proposed
nonlocal models converge to the same local Neumann problem when the
L\'evy measure approches the ``local'' case $\alpha=2$. More
precisely, we consider L\'evy measures $\mu_\alpha$ with densities
$(2-\alp)\frac{g(z)}{|z|^{N+\alpha}}$, where $g$
is a nonnegative bounded function which is $C^1$ in a
neighborhood of $0$ and satisfies $g(0)>0$. In this case we prove that
whatever nonlocal Neumann model we use, the solutions $u_\alp$ converge as
$\alpha\to2$ to the unique viscosity solution of the same limit
problem, namely
\begin{equation}\begin{cases}
		-a\Delta u - b\cdot Du +u = f & \text{in } \Omega\,,\\
		\qquad\dfrac{\partial u}{\partial \mathbf{n}}=0 & \text{in } \partial\Omega\,,
\end{cases}\end{equation}
where $a:=g(0)\frac{|S^{N-1}|}{N}$ and $b:=Dg(0)\frac{|S^{N-1}|}{N}\,.$
\\

\noindent\textsc{Related work --} One of the first papers on the
subject was  Menaldi and Robin \cite{MR}. In that paper stochastic
differential equations for reflection problems are solved in the case
of diffusion processes with inside jumps, i.e. when there  are no
jumps outside the domain. They use the method of ``penalization on the domain''
inspired by Lions, Menaldi and Sznitman \cite{LMS}.

In model (a) (the censored case), any outside jump of the underlying
process is cancelled (censored) and the process is restarted
(resurrected) at the origin of that jump. We refer to
e.g. \cite{BBC,FOT,GQY2,NJ:Book} for more details on such processes. The
process can be constructed using the Ikeda-Nagasawa-Watanabe piecing together
procedure \cite{BBC,FOT,NJ:Book},
as Hunt processes associated to some Dirichlet forms \cite{BBC,GQY2},
or via the Feynman-Kac transform involving the killing measure
\cite{BBC}. In particular, the underlying processes in this paper are
related to the censored stable processes of Bogdan et al. \cite{BBC}
and the reflected $\alp$-stable process of Guan and Ma
\cite{GQY2}. But note that we essentially only construct the
generators and not yet the processes themselves. On the technical
side, we use viscosity solution methods, while \cite{BBC,GQY2} use the
theory of Dirichlet forms and potential theory. Our assumptions are
also different, e.g. with our arguments we treat more general measures and
we have the potential to treat non-linear problems, while \cite{BBC,GQY2} e.g.
treat much more general domains.

Let us also mention that the "natural" Neumann boundary condition for
the reflected $\alp$-stable process of Guan and Ma \cite{GQY2} is
slighly different from the one we consider here. They claim that the
boundary condition arising through the variational formulation and
Green type of formulas is
  $$\lim_{t\to 0} t^{2-\alpha} \frac{\partial u}{\partial x_N}(x+t e_N) =0.$$
  This formula allows the normal derivative $\frac{\partial
    u}{\partial x_N}$ to explode less rapidly than $|x_N|^{\alpha-2}$.

In model (b) outside jumps are stopped where the jump path hits the
boundary, and then the process is restarted there. Model (c) is close
to the approach of Lions-Sznitman in \cite{LS}, and here
outside jumps are immediately projected to the boundary along the
normal direction. This type of models will be thoroughly investigated
in the forthcoming paper \cite{BGJ} by three of the authors, but this time in
the setting of fully non-linear equations set in general domains. Note
that model (b) and (c) coincide in one dimension, i.e. when $N=1$.
Finally, in model (d), outside jumps are mirror reflected at the
boundary.  This is intuitively the natural way  of understanding
a "reflection", but the model may be problematic to handle in general
domains due to the possibility of multiple reflections. E.g. it is not
clear to us if it correspond to an underlying Markov process in a
general domain.

 To the best of our knowledge, processes with generators of the form
 (b)--(d) have not been considered before. E.g. the works of Stroock
 \cite{STR} and Taira \cite{TAI1,TAI:Book} seem
 not to cover our cases because their integrodifferential operators
 involve measures and jump vectors $\eta$ that are more regular than
 ours. Moreover in these works, it is the measure and not $\eta$ that
 prevents the process to jump outside $\Omega$.

In the case of symmetric $\alp$-stable processes (a subordinated
 Brownian motion), our formulation follows after a ``reflection'' on
 the boundary. So such processes can be constructed from a Brownian
 motion by first subordinating the process and then reflecting it.
 Another possible way to construct a ``reflected'' process in this
 case would be to reflect the Brownian motion first and then
 subordinate the reflected process. A related approach is described
 e.g. in Hsu \cite{HSU} where  pure jump processes like L\'evy
 processes are connected via the Dirichlet-Neumann operator to the trace at the
 boundary of a Reflected Brownian Motion in one extra space dimension
 $\Omega\times\R_+$.
%
An analytic PIDE version of this approach is introduced by Caffarelli
and Silvestre in \cite{CC} in order to obtain Harnack inequalities for
solutions of integrodifferential equations, and then these ideas have
been used by many authors since.

  Finally we mention the more classical work of Garroni and Menaldi
  \cite{GM}, where a large class of uniformly elliptic integro
  differential equations are considered. There are two main
  differences with our work: (i) In \cite{GM} the principle part of
  all equations is a local non-degenerate 2nd order term. This allows
  the non-local terms to be controlled by local terms (the solution and
  its 1st and 2nd derivatives) via interpolation inequalities, and the
  local $W^{2,p}$ and $C^{2,\alp}$ theories can therefore be extended to this
  nonlocal case. In our paper, it is the non-local terms that are the
  principal terms, and interpolation is not available. In addition, most
  of our results can be extended to degenerate problems. (ii) In
  \cite{GM} Dirichlet type problems are considered, and the authors
  have to assume extra decay properties of the jump vector $\eta$ near
  the boundary, conditions that are not satisfied in our Neumann
  models.

\section{Assumptions and Definition of solutions}\label{ass-def}

In this section we state the assumptions on the problem
\eqref{E} that we will use in the rest of the paper, give the
definition of solutions for \eqref{E}, and show that the quantities in
this defintion is well-defined under our assumptions.

In this paper we let  $Du(x)$ and $D^2u(x)$ denote the gradient and
hessian matrix of a function $u$ at $x$. We also define
$P(x,z)=x+\eta(x,z)$, and then we can state our assumptions as follows:
\begin{itemize}
\item[$ (\mathrm{H}_f)\,$] $f\in C_b(\domb)$.\medskip
	\item[$(\mathrm{H}_\mu)$] The measure $\mu$ is the sum of two nonnegative Radon measures $\mu_*$ and $\mu_\#$,
          $$\mu=c\mu_*+\mu_\#,$$
 where $c$ is either $0$ or $1$, $\mu_*$ is a symmetric measure satisfying \eqref{Levy} and
	\begin{align*}
		 \int_{|z|<1}|z|\,d\mu_*=\infty, \quad\text{and}\quad\int_{\R^N}(1\wedge|z|)\,d\mu_\#<\infty.
	\end{align*}
\item[$(\text{H}_\eta^0)$] Neuman problem: $P(x,z)_N=x_N+\eta(x,z)_N>0$ for any $x,z$ and
	$$\eta(x,z)=z\quad\text{for any } x_N+z_N>0.$$
\item[$(\text{H}_\eta^1)$] At most linear growth of the jumps: there exists $c_\eta>0$ such that
	$$	|\eta(x,z)|\leq c_\eta |z|\text{ for any $x,z$}.$$
\item[$(\text{H}_\eta^2)$] Antisymmetry with respect to the $z'$-variables: for any $i=1,\dots,(N-1)$,
	$$	\eta(x,\sigma_iz)_i=-\eta(x,z)_i\text{ where }\sigma_iz=(z_1,\dots,-z_i,\dots,z_N)\,.$$
\item [$(\text{H}_\eta^3)$] Weak continuity condition:
	$$\eta(y,z)\to \eta(x,z)\quad \text {$\mu$-a.e.}\quad \text {as } y\to x.$$
\item[$(\text{H}_\eta^4)$] Continuity in the $x'$-variable:
$$	|\eta(x',s,z)'-\eta(y',s,z)'|\leq C|z||x'-y'|\quad \text{for
  any $x',y',z$ and any $s>0$}.$$
\item[$(\text{H}_\eta^5)$] Non-censored cases: Contraction in the $N$-th direction:
	$$|P(x,z)_N-P(y,z)_N|\leq|x_N-y_N|.$$
\item[$(\text{H}_\eta^6)$] Censored case: For all $z\neq0$ and $x\in\dom$,
	$$\eta(x,z)=\begin{cases}z & \text{if }
 x_N+z_N\geq0,\\
0 & \text{otherwise}.
 \end{cases} $$
\end{itemize}

\begin{remark}
Assumption $(\mathrm{H}_\mu)$ means that we can decompose $\mu$ into a
sum of a singular {\em symmetric} L\'evy measure and a not so singular
L\'evy measure. Symmetric here means that
$ \int_{\R^N} \chi(z)\,d\mu = 0$ for any odd  $\mu$-integrable function $\chi$.
This assumption covers the stable, the tempered stable, and the
larger class of self-decomposable processes in $\R^N$,
cf. chapter 1.2 in \cite{A:Book}. In all these cases the L\'evy measures satisfy
$$\frac{d\mu}{dz}=\frac{g(z)}{|z|^{N+\alp}} \quad\text{for}\quad
\alp\in(0,2),$$
and $(\mathrm{H}_\mu)$ holds with $c=0$ if $\alp\in(0,1)$,
while if $\alp\in[1,2)$ and $g$ is Lipschitz in $B_1(0)$ and bounded,
then $(\mathrm{H}_\mu)$ holds with $c=1$. In the last case we may take
$$\frac{d\mu_{*}}{dz}=\frac{h(z)}{|z|^{N+\alp}}\quad\text{and}\quad
\frac{d\mu_{\#}}{dz}=\frac{g(z)-h(z)}{|z|^{N+\alp}}\quad
\text{for}\quad h(z):=\min_{|y|=|z|}g(y),$$
and note that $h$ is symmetric and $g-h$ is nonnegative. More
generally, we can consider
measures where $\frac{d\mu}{dz}=g(z)\frac{d\mu_*}{dz}$
and $\mu^*$ is symmetric.
\end{remark}

\begin{remark}
The cases $\mathrm{(a)}$, $\mathrm{(b)}$, $\mathrm{(c)}$, and $\mathrm{(d)}$
  mentioned in the introduction satisfy assumptions
  $(\text{H}_\eta^i)$ for $i=0,1,2,3,4$, where in fact $(\text{H}_\eta^4)$
  holds with $C=0$. Assumption $(\text{H}_\eta^5)$ holds
  except in case $\mathrm{(a)}$, and case $\mathrm{(a)}$ is
  equivalent to $(\text{H}_\eta^6)$. Note  that $\eta$
  is continuous in $x$
  for $z\neq0$ in  $\mathrm{(b)}$, $\mathrm{(c)}$, and $\mathrm{(d)}$, while
  in $\mathrm{(a)}$, $\eta$ is continuous except on the codimension
  $1$ hypersurface $\{z_N=x_N\}$.
\end{remark}

Now we will define generalized solutions in the viscosity sense, and
to do that we need the following notation:
$$I[\phi]=I_\delta[\phi]+I^\delta[\phi]\,,$$
where
$$
I^\delta[\phi]=\int_{|z|\geq\delta}\phi(x+\eta(x,z))-\phi(x)\,d\mu(z)\,.
$$
The $I^\delta$-term is well-defined for any bounded function $\phi$. For
the $I_\delta$-term there are two cases, depending on whether $c=0$ or
$1$ in $(\text{H}_\mu)$. If $c=0$, a Taylor expansion shows that
$I_\delta[\phi](x)$ is well-defined for $\phi\in C^1$ and $x\in\dom$. If $c=1$,
and the measure $\mu$ is very singular, we add and subtract a
compensator and write
\begin{align*}
I_\delta[\phi](x)=\tilde I_\delta[\phi](x)+\mathrm{P.V.}
        \int_{|z|<\delta}D\phi(x)\cdot\eta(x,z)\,d\mu(z),
\end{align*}
for
$$\tilde I_\delta[\phi](x):= \int_{|z|<\delta}\phi(x+\eta(x,z))-\phi(x)-D\phi(x)\cdot\eta(x,z)\,d\mu(z).$$
By the  $C^2$-regularity of $\phi$, these two terms will be
well-defined --
see Lemma \ref{cor:compensator} below. Note that this results is non-trivial
because the compensator is not well defined in general!

\begin{definition}
\label{def1} \label{defin} Assume that $\mathrm{(H_\mu)}$,
$(\mathrm{H}_\eta^i)$ for $i=0,1,2$ hold.

\noindent(i) A bounded usc function $u$ is a viscosity {\em subsolution} to \eqref{E}
if, for any test-function $\phi\in C^2(\R^N)$ and  maximum point $x$ of $u-\phi$
in $B(x,c_\eta\delta)\cap\domb$,
\begin{align*}
& F(x,u(x),I_\delta[\phi]+I^\delta[u])\leq 0 & \quad\text{if }
x\in\dom\quad\text{and}\\
\text{either}\quad&
F(x,u(x),I_\delta[\phi]+I^\delta[u]) \leq 0 &\quad\text{if } x\in\del\dom \text{ and }c=0,\\
\text{or}\quad&
- \frac{\del\phi}{\del x_N}(x)\leq 0 &\quad\text{if } x\in\del\dom \text{ and }c=1.
\end{align*}
\noindent(ii) A bounded lsc function $v$ is a viscosity {\em supersolution} to
\eqref{E} if, for any test-function $\phi\in C^2(\R^N)$ and  minimum point $x$ of $v-\phi$
in $B(x,c_\eta\delta)\cap\domb$,
\begin{align*}
& F(x,v(x),I_\delta[\phi]+I^\delta[v])\geq 0 & \quad\text{if }
x\in\dom\quad\text{and}\\
\text{either}\quad&
F(x,v(x),I_\delta[\phi]+I^\delta[v]) \geq 0 &\quad\text{if } x\in\del\dom \text{ and }c=0,\\
\text{or}\quad&
- \frac{\del\phi}{\del x_N}(x)\geq 0 &\quad\text{if } x\in\del\dom \text{ and }c=1.
\end{align*}
\noindent(iii) A viscosity {\em solution} is both a sub- and a supersolution.
\end{definition}
\begin{remark}
The constant $c_\eta$ is defined in $(\mathrm{H}_\eta^1)$. If
$u$ and $\phi$
are smooth and $x$ is a maximum point of $u-\phi$ over
$B(x,c_\eta\delta)\cap\domb$, then by $(\mathrm{H}_\eta^1)$,
$$u(x)-\phi(x)\geq u(x+\eta(x,z))-\phi(x+\eta(x,z)) \quad\text{for
  all}\quad |z|<\delta.$$
If we rewrite this inequality and integrate, we find formally that
$I_\delta[u](x)\leq I_\delta[\phi](x)$. Lemma \ref{cor:compensator}
below makes this computation rigorous. From this inequality it is easy to
prove that classical (sub)solutions of \eqref{E} are viscosity
(sub)solutions. Moreover, smooth viscosity (sub)solutions
are classical (sub)solutions (simply take $\phi=u$).
\end{remark}
\begin{remark}
In general to pose boundary value problems in the viscosity
sense, one requires that either the minimum of the equation and the
boundary condition is nonpositive or the maximum of the equation and the
boundary condition is nonnegative. Here this is not the case and for a
natural reason. If the measure is very singular (and $c=1$) then
the equation cannot hold on the boundary and the inequality holds just
for the boundary condition. In the $c=0$ case, on the contrary, the
equation will hold up to the boundary and the boundary condition can
not be imposed in general.
In other words, we only end up with a Neumann boundary condition if
$c=1$, i.e. the measure has a ``strong'' singular part $\mu_*$. In
this case the intensity of small jumps is so strong that the
jump-reflection mechanisms, e.g. as in $(a)$ -- $(d)$, are not enough
to prevent the process from ``diffusing'' onto the boundary, and we
need to add a reflection process at the boundary to keep the process
inside (just as in the case of Brownian motion). We also note
that the symmetry of $\mu_*$ is a natural condition in order to
obtain Neumann and not oblique derivative boundary
conditions, cf. Lemma \ref{gamman} and proof.
\end{remark}

We now prove that $I_\delta[\phi]$ is well-defined for $\phi\in C^2$.

\begin{lemma}\label{cor:compensator}
	Assume $\mathrm{(H_\mu)}$ and $(\mathrm{H}_\eta^i)$ for
        $i=0,1,2$, and let $x\in\dom$, $\phi\in C^2$, and
        $\delta>0$. Then $I_\delta[\phi](x)$ is well-defined since
\begin{align*}
&I_\delta[\phi](x)=\tilde I_\delta[\phi](x)+\mathrm{P.V.}
        \int_{|z|<\delta}D\phi(x)\cdot\eta(x,z)\,d\mu(z),\\
\intertext{and the compensator term satisfies}
&\mathrm{P.V.}\int_{|z|<\delta}D\phi(x)\cdot\eta(x,z)\,d\mu(z)\\
&=\int_{|z|<\delta}D\phi(x)\cdot\eta(x,z)\,d\mu_\#(z)+c\int_{x_N<|z|<\delta}D\phi(x)\cdot\eta(x,z)\,d\mu_*(z).
\end{align*}
Moreover, there is $R=R(x,\eta)>0$ such that
$$I_\delta[\phi](x)=o_\delta(1)\|\phi\|_{C^2(B_R(x))}\quad\text{as}\quad
\delta\ra0. $$
\end{lemma}

In the following, we often drop the $\mathrm{P.V.}$ notation for such
integrals  since they may be expressed in terms of converging integrals. Note
that the integral over $\{x_N<|z|<\delta\}$
need not vanish since this regions exceeds the boundary and hence
$\eta(x,z)$ will not be odd there.

To prove Lemma \ref{cor:compensator}, we need the following result.
\begin{lemma}\label{lem:mu.sharp}
	Assume $(\mathrm{H}_\mu)$ and $(\mathrm{H}_\eta^i)$ for
        $i=0,1,2$, and let $x_N>0$, $\rho\in(0,x_N)$, and $v\in\R^N$
        be a fixed vector.
	\begin{itemize}
	\item[$(i)$] For $r\in(0,\rho)$,
	$\int_{r<|z|<\rho} v\cdot \eta(x,z)\,d\mu(z)=
	\int_{r<|z|<\rho} v\cdot \eta(x,z)\,d\mu_\#(z)$, and
	$$\mathrm{P.V.}\int_{|z|<\rho} v\cdot \eta(x,z)\,d\mu(z)=
	\int_{|z|<\rho} v\cdot \eta(x,z)\,d\mu_\#(z)\,.$$
	\item[$(ii)$] For $r\in(0,1)$,
	$\int_{r<|z|<\delta} v'\cdot \eta(x,z)'\,d\mu(z)=
	\int_{r<|z|<\delta} v'\cdot \eta(x,z)'\,d\mu_\#(z)$, and
	$$\mathrm {P.V.}\int_{|z|<\delta} v'\cdot \eta(x,z)'\,d\mu(z)=
	\int_{|z|<\delta} v'\cdot \eta(x,z)'\,d\mu_\#(z)\,,$$
	\item[$(iii)$] For $r\in(0,1)$,
	$$\int_{r<|z|<\delta} \eta(x,z)_N\,d\mu_*(z)\geq
	\int_{r<|z|<\delta\atop z_N>x_N} (z_N-x_N)\,d\mu_*(z)\geq0\,.$$
	\end{itemize}
\end{lemma}
\begin{proof}
 (i) If $|z|<\rho<x_N$, then $\eta(x,z)=z$ by $(\text{H}_\eta^0)$. Hence
  $\eta$ is odd with respect to the $z$ variable, and the integral
  with respect to the symmetric part $c\mu_*$ is zero.
  Passing to the limit as $r\to0$ and using the integrability
 of $\mu_\#$ along with $(\text{H}_\eta^1)$ finishes the proof of
  $(i)$.

(ii) Let $\sigma$ be the rotation $\sigma(z',z_N)=(-z',z_N)$. Since
$\mu_*$ is symmetric,
	$$\int_{r<|z|<\delta} v'\cdot \eta(x,\sigma z)'\,d\mu_*(z)=\int_{r<|z|<\delta} v'\cdot \eta(x,z)'\,d\mu_*(z).$$
Other hand, since
        $\eta(x,-z',z_N)'=-\eta(x,z',z_N)'$ by $(\text{H}^2_\eta)$,
         the above integral is zero.  The result on the principal value is
        obtained as in the first case, after letting $r\to0$.

(iii)	Let us decompose
	$$\begin{aligned}
	&\int_{r<|z|<\delta}\eta(x,z)_N\,d\mu_*(z)=\int_{r<|z|<\delta \atop -x_N\leq z_N\leq x_N}
	\eta(x,z)_N\,d\mu_*(z)\\
	&\quad+\int_{r<|z|<\delta \atop z_N>x_N}\eta(x,z)_N\,d\mu_*(z)+\int_{r<|z|<\delta \atop z_N<-x_N}
	\eta(x,z)_N\,d\mu_*(z)\,.
	\end{aligned}$$
	The integral over $-x_{N}\leq
        z_N\leq x_{N}$ vanishes since
	$\eta(y,z)=z$ in this region and $\mu_*$ is symmetric. By
        $(\mathrm{H}_\eta^{0})$ we have that $\eta(x,z)_N\ge -x_{N}$
        if $z_N<-x_{N}$ and $\eta(x,z)=z_N>x_{N}$ if $z_N>x_{N}$.
	Hence by symmetry of $\mu_*$,
	$$\int_{r<|z|<\delta}
        \eta(x,z)_N\,d\mu_*(z)\geq\int_{r<|z|<\delta\atop z_N>x_N}
        (z_N-x_N)\,d\mu_*(z)\geq 0,$$
and the proof is complete.
\end{proof}

\begin{proof}[Proof of Lemma \ref{cor:compensator}]
The expression for $I_\delta$ is obtained by adding and substracting
the compensator term. The first integral in this expression is
well-defined since the integrand is smooth and bounded by the function
        $\frac12|z|^2\max_{B(x,R)}|D^2\phi|$, for $R=\max_{y\in
          B(0,\delta)}|\eta(x,y)|$, which is an $\mu$-integrable function
        over $B(0,\delta)$. Moreover,
        $\int_{0<|z|<\delta}|z|^2\,d\mu(z)=o_\delta(1)$ as
        $\delta\ra0$ since $|z|^2$
        is $\mu$-integrable near $0$.

In the compensator term, the integral with respect to $\mu_\#$ exists by
$\mathrm{(H_\eta^1)}$, while the integral with respect to $\mu_*$ over
$B(0,x_N)$ vanishes by Lemma \ref{lem:mu.sharp}-$(i)$.
	Since $|z|$ is integrable near the origin for $\mu_\#$,
        this term is $|D\phi(x)|o_\delta(1)$ as
        $\delta\to0$.
\end{proof}

\section{Derivation of the boundary value problem - PIDE approach}\label{der-pide}

In this section we derive the boundary value
problems from approximate problems involving a sequence of bounded
positive Radon measures $\mu^k=1_{\{|z|>1/k\}}\mu$ converging to $\mu$.
Assume $\mathrm{(H_\mu)}$ and let
$\mu_\#^k=1_{\{|z|>1/k\}}\mu_\#$ and
$\mu_*^k=1_{\{|z|>1/k\}}\mu_*$, it then  easily follows that
\begin{align*}
{(\mathrm{H}_\mu^1)}&\quad\lim_{k\ra+\infty}\int |z|\wedge 1 \ d\mu_\#^k(z) =
\int |z|\wedge 1 \ d\mu_\#(z),\\
{(\mathrm{H}_\mu^2)}&\quad \lim_{k\ra+\infty}\int |z|^2\wedge 1 \ d\mu_*^k(z) =
\int |z|^2\wedge 1 \ d\mu_*(z),\\
{(\mathrm{H}_\mu^3)}&\quad \lim_{k\ra+\infty}\int |z|\wedge 1 \
d\mu_*^k(z) = \infty.
\end{align*}
The approximation problem we consider is then given by
\begin{align}
\label{En}
 u(x)-I_{\mu_k}[u](x)=f(x)\quad\text{in}\quad\domb,
\end{align}
where, for $\phi \in C_b(\domb)$,
$$
I_{\mu_k}[\phi](x)=\int_{|z|>0}\phi(x+\eta(x,z))-\phi(x)\ d\mu^k(z).
$$

Since the measures $\mu^k$ are bounded, this equation holds in a classical,
pointwise sense. Moreover, it is well-posed in $C_b(\domb)$ and the
solutions $u_k$ are uniformly bounded in $k$:
\begin{lemma}\label{lem_uk}
Assume $(\mathrm{H}_f),\ (\mathrm{H}_\mu),\ (\mathrm{H}_\eta^0),$ and $(\mathrm{H}_\eta^3)$.

\noindent (a) For every $k$, there is a unique pointwise solution
$u_k$ of \eqref{En} in $C_b(\domb)$.

\noindent (b) If $u_k$ and $v_k$ are pointwise sub- and supersolutions of
\eqref{En}, then $u_k\leq v_k$ in $\domb$.

\noindent (c) If $u_k$ is a pointwise solution of \eqref{En}, then  $\|u_k\|_{L^\infty(\dom)}\leq
\|f\|_{L^\infty(\dom)}$.
\end{lemma}
\begin{proof}
(a) Let  $T:C_b(\domb)\to C_b(\domb)$ be the operator defined by
$$Tu:=u-\eps\big(u-I_{\mu_k}[u]-f\big)\,,$$
where $\eps<(1+2\|\mu^k\|_1)^{-1}$ and $\|\mu^k\|_1$ is the total (finite!) mass of the measure $\mu^k$.
Then $T$ is a contraction in the Banach space $C_b(\domb)$ since
$$\begin{aligned}
\|Tu-Tv\|_\infty&\leq(1-\eps)\|u-v\|_\infty+2\eps\|\mu^k\|_1\|u-v\|_\infty\\
&\leq \big(1-\eps(1+2\|\mu^k\|_1)\big)\|u-v\|_\infty\\
&\leq C(k)\|u-v\|_\infty\,,
\end{aligned}$$
and $C(k)<1$. Hence there exists a unique $u_k\in C_b(\domb)$ such
that $Tu_k=u_k$, which is equivalent to \eqref{En}.

(b) If $\sup_{\dom} (u-v)$ is
attained at a point $x\in \domb$, then by the equation
and the easy fact that $I_{\mu_k}[\phi]\leq0$
at a maximum point of $\phi$,
$$\sup_\dom (u-v) = u(x)-v(x)\leq I_{\mu_k}[u-v](x)\leq 0.$$
The general case follows after a standard penalization argument.

(c) Follows from (b) since $\pm \|f\|_{L^\infty(\dom)}$ are sub- and
supersolutions of \eqref{En}.
\end{proof}


The limiting problem can be identified through the half relaxed limit method:

\begin{theorem}\label{limit} Assume $(\mathrm{H}_f)$, $\mathrm{(H_\mu)}$, and
  $(\mathrm{H}_\eta^i)$ for $\ i=0,1,2,3$ hold.
	Then the half-relaxed limit functions
	$$\ol{u}(x)=\limsup_{k\ra+\infty,y\ra x}u_k(y)\quad\text{and}\quad
\ul{u}(x)=\liminf_{k\ra+\infty,y\ra x}u_k(y)\,$$
are respectively sub- and supersolutions of the Neumann boundary problem in the sense of  Definition~\ref{def1}.
\end{theorem}

In the proof we will need the following result whose proof is given at
the end of this section.
\begin{lemma}\label{gamman}
Assume $(\mathrm{H}_\eta^i)$ holds for $i=0,1,2$, $\mathrm{(H_\mu)}$
holds with $c=1$, and let $\delta>0$ and
$\gamma_{\mu_k,r}(x):=\int_{|z|<r}\eta(x,z)\,d\mu^k(z)$. If $y_k\to
x\in \bound$ as $k\to\infty$, then
	$$|\gamma_{\mu_k,\delta}(y_k)|\to\infty\quad\text{and}\quad
	\frac{\gamma_{\mu_k,\delta}(y_k)}{|\gamma_{\mu_k,\delta}(y_k)|}
        \to-\mathbf{n},$$
        where
	$\mathbf{n}=(0,0,\dots,0,-1)$ is an outward normal vector
	of $\del\dom$.
\end{lemma}

\begin{proof}[Proof of Theorem \ref{limit}]
Since the proofs are similar for $\ol u$ and $\ul u$, we only do the
one for $\ol u$. Let $\delta>0$ and $\phi\in C^2$, and assume that
$\ol u-\phi$ has a maximum point $x$ in
$B(x,c_\eta\delta)\cap\domb$. Let us first
consider the case when
$x\in\dom$, i.e. when $x_N>0$. By modifying the test-function, we may
always assume that the maximum  is strict.
By standard arguments, $u_k - \phi$ has
a maximum point $y_k$ in $B(x,c_\eta\delta)$, and when $k\ra+\infty$,
$$y_k\ra x\qquad \text{and}\qquad u_k(y_k)\ra \ol u(x).$$
Let $\delta_k=\delta-|x-y_k|$ and $0<r\leq\delta_k$, and note that
$B(y_k,c_\eta r)\subset B(x,c_\eta\delta)$. Since the max of $(u_k-\phi)$ in
$B(y_k,c_\eta r)$ is  attained at $y_k$, we find that
$$\begin{aligned}
&(I_{\mu_k})_{r}[u_k](y_k):=\int_{|z|<r}u_k(y_k+\eta(y_k,z))-u_k(y_k)\,d\mu^k\\
&\leq
\int_{|z|<r}\phi(y_k+\eta(y_k,z))-\phi(y_k)\,d\mu^k=(I_{\mu_k})_{r}[\phi](y_k).
\end{aligned}$$
Hence, since $u_k$ is a pointwise solution of \eqref{En}, we find for
all $0<r\leq\delta_k$,
\begin{align*}
u_k(y_k)-(I_{\mu_k})_{r}[\phi](y_k)-(I_{\mu_k})^{r}[u_k](y_k)
\leq f(y_k),
\end{align*}
where
$(I_{\mu_k})^r[u_k](x):=\int_{|z|\geq r}u_k(x+\eta(x,z))-u_k(x)\,d\mu^k(z)$.

We want to pass to the limit in this
equation and consider first the $(I_{\mu_k})^r$-term. By the definition of $\ol
u$ and $(\mathrm{H}_\eta^3)$,
$$\limsup_{k\ra+\infty} u_k(y_k+\eta(y_k,z))\leq \ol u(x+\eta(x,z))\quad \text{for
  a.e. }z.$$
Hence, since we integrate away from the singularity of
$\mu$, we can use Fatou's lemma and $(\mathrm{H}_\mu^1)$ and
$(\mathrm{H}_\mu^2)$ to show that
\begin{align*}
&\limsup_{k\ra\infty}(I_{\mu_k})^r[u_k](y_k)\leq \int_{|z|>r}
\ol{u}(x+\eta(x,z))-\ol{u}(x)\, d\mu(z)=I^r[\ol u](x).
\end{align*}

To pass to the limit in the $(I_{\mu_k})_r$-term, we have to write it as
\begin{align*}
&(I_{\mu_k})_r[\phi](y_k)=\\
&\underbrace{\int_{|z|<r}\phi(x+\eta(y_k,z))-\phi(y_k)
 -D\phi(y_k)\cdot \eta(y_k,z)
\,d\mu^k(z)}_{=(\tilde I_{\mu_k})_r[\phi](y_k)} +
\gamma_{\mu_k,r}(y_k)\cdot D\phi(y_k),
\end{align*}
where $\gamma_{\mu_k,r}(x):=\int_{|z|<r}\eta(x,z)\,d\mu^k(z).$
For $|z|<r$, a Taylor expansion then yields
$$\big|\phi(y_k+\eta(y_k,z))-\phi(y_k)-D\phi(y_k)\cdot\eta(y_k,z)\big|\leq
\|D^2\phi\|_{L^\infty(B(x,c_\eta r))}|z|^2.$$
Hence by $(\mathrm{H}_\eta^1)$, $(\mathrm{H}_\eta^3)$,
$(\mathrm{H}_\mu^1)$ and $(\mathrm{H}_\mu^2)$, we can use the
Dominated Convergence Theorem to
show that
$$
 (\tilde I_{\mu_k})_r[\phi](y_k)\ra
\int_{|z|<r}\phi(x+\eta(x,z))
-\phi(x)-\eta(x,z)D\phi(x)\ d\mu(z)=\tilde I_r[\phi](x).
$$
Next, by Lemma~\ref{cor:compensator},
$$\gamma_{\mu_k,r}(y_k)=\int_{|z|<r}\eta(y_k,z)\,d \mu^k_\#(z)
+c\int_{y_{k,N}\leq |z|<r}\eta(y_k,z)\,d \mu^k_*(z),$$
where the last integral is understood to be zero if
$y_{k,N}>r$. Note that since
$y_{k,N}\to x_N>0$, the domain of integration of the $\mu_*$-integral
is always bounded away from $z=0$ when $k$ is big.
Along with $(\mathrm{H}_\eta^3)$ and $(\text{H}_\mu^2)$,
this allows us to pass to the limit in the $\mu_*$-integral using
the Dominated Convergence Theorem. Similarly, we may pass to the limit
in the $\mu_\#$-integral by $(\mathrm{H}_\eta^1)$,
$(\mathrm{H}_\eta^3)$ and $(\text{H}_\mu^1)$. We find that
\begin{align*}
\gamma_{\mu_k,r}(y_{k,N})\to\gamma_r(x)&:=\int_{|z|<r}z\,d \mu_\#(x)+
c\int_{x_N\leq |z|<r}\eta(x,z)\,d \mu_*(x)\\
&=\mathrm{P.V.}\int_{|z|<r}\eta(x,z)\,d \mu(x)=:\gamma_r(x),
\end{align*}
where we used Lemma~\ref{cor:compensator} again. Hence we can conclude that
$$\lim_{k\ra\infty}(I_{\mu_k})_{r}[\phi](y_k)=\tilde
I_r[\phi](y_k)+\gamma_r(x)\cdot D\phi(x)=I_r[\phi](x).$$
Since $\delta_k\ra\delta$, we end up with the following limit equation,
$$\ol{u}(x)-I_{r}[\phi](x)-I^r[\ol u](x)\leq f(x)$$
for every $0<r<\delta$. Using the Dominated Convergence Theorem again,
we send $r\ra\delta$ and obtain the subsolution condition for \eqref{E} at
the point $x\in\dom$.

The second part of the proof is to consider the case of when
$x\in\del\dom$, i.e the case when $x_N=0$. We first do it in the
case $c=1$. By adding, subtracting, and divinding by terms, we may
rewrite the subsolution condition as
\begin{align*}
&\frac{u_k(y_k)-(\tilde I_{\mu_k})_{\delta}[\phi](y_k)-(I_{\mu_k})^{\delta}[u_k](y_k)-f(y_k)}{|\gamma_{\mu_k,\delta}(y_k)|}-\frac{\gamma_{\mu_k,\delta}(y_k)\cdot
  D\phi(y_k)}{|\gamma_{\mu_k,\delta}(y_k)|}\leq 0.
\end{align*}
By Lemma~\ref{gamman}, $|\gamma_{\mu_k,\delta}(y_k)|\to\infty$, and
since $u_k$ and $f$ are uniformly bounded,
$$\frac{u_k(y_k)}{|\gamma_{\mu_k,\delta}(y_k)|}\,,\quad \frac{f(y_k)}{|\gamma_{\mu_k,\delta}(y_k)|}\,,\quad
\frac{(I_{\mu_k})^{\delta}[u_k](y_k)}{|\gamma_{\mu_k,\delta}(y_k)|},$$
all converge to zero. The same is true for
$$\frac{(\tilde I_{\mu_k})_{\delta}[\phi](y_k)}{|\gamma_{\mu_k,\delta}(y_k)|}$$
since the integrand of the numerator is controlled by
$\|D^2\phi\|_\infty|z|^2{1}_{\{|z|<\delta\}}$ and $\mu^k$ satisfies
$(\mathrm{H}_\mu^1)$. Using Lemma~\ref{gamman} again, we have
$\gamma_{\mu_k,\delta}(y_k)/|\gamma_{\mu_k,\delta}(y_k)|\to
-\mathbf{n}$, so that we may go to the limit in the above inquality to find that
$$-\frac{\partial \phi}{\partial x_N}(x)=\frac{\partial \phi}{\partial \mathbf{n}}(x)\leq 0\,.$$

In the case when $c=0$, the measure $\mu=\mu_\#$ which
less singular than $\mu_*$. The same line of
arguments as in the proof for $x\in\dom$ (much easier this time) now
shows that the equation holds at $x\in\del\dom$.
\end{proof}

\begin{proof}[Proof of Lemma~\ref{gamman}]
First note that by Lemma \ref{lem:mu.sharp} with $y_k$ instead of $x$ and $\mu^k$
        instead of $\mu$,
	$$\gamma_{\mu_k,\delta}(y_k)'=\int_{|z|<\delta}\eta(y_k,z)'\,d\mu^k(z)=\int_{|z|<\delta}\eta(y_k,z)'\,d\mu_\#^k(z)\,,$$
which remains uniformly bounded in $k$ because of $(\mathrm{H}_\eta^1)$ and our assumption on $\mu_\#$.
	Since $y_{k,N}\to x_N=0$,
	we can assume that $0\leq y_{k,N}<\delta$, and by Lemma
        \ref{lem:mu.sharp},
	$$(\gamma_{\mu_k,\delta})_N(y_k)=\int_{|z|<\delta}\eta(y_k,z)_N\,d\mu_\#^k(z) + \int_{y_{k,N}<|z|<\delta}\eta(y_k,z)_N\,d\mu_*^k(z)\,.
	$$
	As above, the first integral is uniformly
        bounded as $k\to\infty$. For the second one, we send $r\to0$ in
	Lemma \ref{lem:mu.sharp}-$(iii)$ to find that
	\begin{equation}\label{eq:est.eta}
	\int_{|z|<\delta}\eta(y_k,z)_N\,d\mu_*^k(z)\geq \int_{|z|<\delta \atop y_{k,N}<z_N}(z_N-y_{k,N})\,d\mu_*^k(z)\geq0,
	\end{equation}
	and, since $y_{k,N}\to0$, we can then use Fatou's lemma to show that
	$$\int_{|z|<\delta\atop z_N>0} z_N\,d\mu_*(z)\leq\liminf_{k\to\infty}
	\int_{|z|<\delta \atop y_{k,N}<z_N}(z_N-y_{k,N})\,d\mu_*^k(z)\,.$$
	Applying symmetry of the measure $\mu_*$ twice, we are lead to
	$$\int_{|z|<\delta\atop z_N>0} z_N\,d\mu_*(z) = \frac 1 2 \int_{|z|<\delta} |z_N|\,d\mu_*(z)=
	\frac 1 {2N} \int_{|z|<\delta}\sum_{i=1} ^N|z_i|\,d\mu_*(z),$$
so by taking $(\mathrm{H}_\mu^3)$ into account, there is a constant $C=C(N)>0$ such that
	$$\int_{|z|<\delta\atop z_N>0} z_N\,d\mu_*(z) \ge  \frac C {2N} \int_{|z|<\delta} |z| \,d\mu_*(z)=\infty.$$

    Hence we have proved that
    $(\gamma_{\mu_k,\delta})_N(y_k)\to\infty$ as $k\to\infty$, and if
    we use that $\big(\gamma_{\mu_k,\delta}\big)'(y_k)$ is uniformly
    bounded, we see that
	$$\frac{\gamma_{\mu_k,\delta}(y_k)}{|\gamma_{\mu_k,\delta}(y_k)|}=\Big(\frac{(\gamma_{\mu_k,\delta})'(y_k)}{|\gamma_{\mu_k,\delta}(y_k)|}
	\,,\, \frac{(\gamma_{\mu_k,\delta})_N(y_k)}{|\gamma_{\mu_k,\delta}(y_k)|} \Big)\longrightarrow (0,0,\cdots,0,1)=-\mathbf{n}.$$

\end{proof}

\section{Comparison in non-censored cases}
\label{sec:nonsens}
In this section we prove a comparison result
for the non-censored cases, i.e. under assumption
$(\mathrm{H}_\eta^i)$ for $i=0,\dots,5$. These assumptions covers all
the examples given in the introduction, except example (a) -- the
censored case. As a
conseqence of the comparison result and the results of the previous
sections, we also obtain well-posedness for \eqref{E}. The comparison
result is the following:

\begin{theorem}\label{fleas}
Assume $(\mathrm{H}_\mu)$, $(\mathrm{H}_f)$, and
$(\mathrm{H}_\eta^i)$ hold for $i=0,1,2,3,4,5$.
    Let $u$ be a bounded usc subsolution of (\ref{E}) with data $f \in C_b(\R^N)$,
    $v$ be a bounded lsc supersolution of (\ref{E}) with data $g \in C_b(\R^N)$
    such that $f\leq g$ in $\domb$. Then $u\leq v$ on $\dom$.
\end{theorem}

From this result it follows that the half-relaxed limits in Theorem
\ref{limit} satisfy $\ol u\leq\ul u$ in $\domb$. Since the opposite
inequality is always satisfied, this means that $u:=\ol u =\ul u$
is solution of \eqref{E} according to Definition
\ref{def1}. Uniqueness and continuous dependence (on $f$) follows from
Theorem \ref{fleas} by standard arguments and we have the following result:

\begin{corollary}
Under the assumptions of Theorem \ref{fleas}, there exits a unique
viscosity solution of \eqref{E} depending continuously on $f$.
\end{corollary}

\begin{proof}[Proof of Theorem \ref{fleas}]
 We argue by contradiction assuming that $M:=\sup_{\domb}(u-v)>0$.
We provide the full details only when $c=1$. The case $c=0$ is far
simpler since the equation then holds even on the boundary.

    To get a contradiction, we first introduce the function
	$$\Psi_{R}(x):=u(x)-v(x)-\psi_R(x,x)\,,$$
    where $\psi_R$ is a localisation term which makes the max attained:
    $$\psi_R(x,y)=\psi(|x_N|/R)+\psi(|y_N|/R)+\psi(|x'|/R)+\psi(|y'|/R)\,,$$
	with $\psi$ a smooth function such that
	$$\psi(s)=\begin{cases}
		0 & \text{for }0\leq s<1/2\,,\\
		\text{increasing} & \text{for } 1/2\leq s<1\,,\\
		2(\|u\|_\infty+\|v\|_\infty + 1) & \text{for }s\geq1\,.\\
	\end{cases}$$
    Of course, $M_{R}:=\max\Psi_{R} \to M$ as $R\to\infty$ so that
    for $R$ big, $\max\Psi_R>0$. Then there are two cases for such $R>0$:
\smallskip

    \noindent\textbf{(a)} either there exists a maximum point $\bar{x}_{R}$ of $\Psi_{R}$
    which is not located on the boundary. In this case the proof is quite classical:
    we use the doubling of variables plus a localisation
    term around $\bar{x}_{R}$ by considering the max of
	$$u(x)-v(y)-\frac{|x'-y'|^2}{\eps'^2}-\frac{|x_N-y_N|^2}{\eps_N^2}
	-\psi_R(x,y)-\sigma|x-\bar{x}_{R}|^4\,,$$
where $\sigma >0$ is small. The localization term $\sigma|x-\bar{x}_{R}|^4$
is chosen so that $\bar{x}_{R}$ is the unique maximum point of
$x\mapsto \Psi_{R}(x) -\sigma|x-\bar{x}_{R}|^4$, and by choosing
$\sigma$ small enough, the contribution of this function in the
integral term is small\footnote{In the viscosity
  inequalities, the various penalization terms are only integrated
  near the origin, e.g. in $\{|z|<\delta\}$. Therefore we do
  not need to worry about the integrability at infinity of $|x|^2$ and
  $|x|^4$ with respect to the measure $\mu$.}.

 Hence the maximum points $(\bar{x},\bar{y})$ of the above test-function converges necessarely to $(\bar{x}_{R},\bar{x}_{R})$ as $\eps', \eps_N\to 0$ and they are also bounded away from the boundary if $\eps', \eps_N$ are small enough. This property implies that we can use directly the equation and obtain $\max\Psi_R\leq0$, which is a contradiction.
\smallskip

    \noindent\textbf{(b)} or any maximum point $\bar{x}_{R}$ is located on the boundary. In this case we
    use the doubling of variables plus some extra term to push the points inside (see below)${}^1$:
    $$\Phi_{\eps',\eps_N,\nu,R}(x,y):=u(x)-v(y)-\frac{|x'-y'|^2}{\eps'^2}-\frac{|x_N-y_N|^2}{\eps_N^2}
	-\psi_R(x,y)+\mathrm{d}_\nu(x_N)+\mathrm{d}_\nu(y_N)\,.$$
    In this case we can assume without loss of generality that the maximum points $\bar{x},\bar{y}$ are
    always such that $0<\bar{x}_N,\bar{y}_N<1$, whatever $\eps_N,\eps',\nu>0$ are.

    Note that in both cases we take two distinct real parameters $\eps',\eps_N>0$ in order to take
    into account the special contraction property in the $N$-th direction, see $(\mathrm{H}_\eta^5)$.
    Now, since case (a) is rather standard, we only provide a proof of case (b) which is more involved.

    The term $\mathrm{d}_\nu$ plays the role of
    a distance to the boundary of the domain; such term is usual in classical Neumann
    proofs in order to prevent the maximum points to be on the boundary. More precisely,
   	for $\nu>0$, we take $\mathrm{d}_\nu(\cdot)=\nu \mathrm{d}(\cdot)$ where $d$ is a smooth function such that
	$$\mathrm{d}(s)=\begin{cases}
		s & \text{for }0\leq s<1/2\,,\\
		\text{increasing} & \text{for } 1/2\leq s<1\,,\\
		1 & \text{for }s\geq1\,.\\
	\end{cases}$$


	Let us note that if $0<\nu<1$ and $R\gg1$ are fixed,
	then $\Phi:=\Phi_{\eps',\eps_N,\nu,R}\leq0$ for $|x|,|y|$ large enough,
	while, by choosing $x=y$ in a suitable way by taking into account the fact that $M>0$,
	we have $\Phi_{\eps',\eps_N,\nu,R}(x,x)>0$ for $\nu$ small; hence the maximum of $\Phi$ is attained
	at some point $(\bar{x},\bar{y})\in\bar{\dom}^2$, that we still denote by $(x,y)$ for simplicity.

  After proving that the points $\bar x,\bar y$ are inside $\dom$, we are going to first let $\eps_N\to0$,
	then $\eps'\to0$, then $\nu\to0$ and finally $R\to\infty$.
	Because of this use of parameters, we have
    $$M_{\eps_N,\eps',\nu,R}:=\max\Phi_{\eps_N,\eps',\nu,R}\to M>0\,.$$
    In particular, this implies that
    $$x_N-y_N = O(\eps_N)\,,\ x'-y'=O(\eps')\,,\
    \frac{|x'-y'|^2}{\eps'^2}=o_{\eps_N,\eps'}(1)\,,$$
    where the $O(\eps_N),O(\eps')$ are uniform with respect to all the parameters, and
    the $o_{\eps_N,\eps'}(1)$ means precisely that after passing to the limit as $\eps_N\to0$,
    we are left with a quantity which is an
    $o_{\eps'}(1)$. Also note that
    $$u(x)-v(y)=M+o_{\eps_N,\eps',\nu,R}(1)\,,$$
where the order of the parameters is important as explained above.
\smallskip

	\noindent\textbf{Step 1 --} \textit{Pushing the points inside.}\\
    We denote by
    $$\varphi(x,y):= \frac{|x'-y'|^2}{\eps'^2}+\frac{|x_N-y_N|^2}{\eps_N^2}
	+\psi_R(x,y)-\mathrm{d}_\nu(x_N)-\mathrm{d}_\nu(y_N)\; ,$$
    where we have dropped the parameters for the sake of simplicity of notations.

    In this step, we prove that the $F$-viscosity inequalities hold for $u$ and $v$.
    According to Definition~\ref{def1}, this is clearly the case if $c=0$ since these viscosity
    inequalities hold even if the maximum or minimum points are on the boundary.

    In the $c=1$ case, let us assume that the maximum point $(x,y)$ is such that $x_N=0$,
    then $x$ is a (global) maximum point of the function $z \mapsto u(z)-v(y)-\varphi(z,y)$ and,
    thanks to Definition~\ref{def1}, we should have $-\frac{\partial \varphi}{\partial x_N}(x,y)\leq 0$.
    But, recalling that $\frac{\partial \psi_R}{\partial x_N}$ is zero in a neighborhood of the boundary, we have
    $$-\frac{\partial \varphi}{\partial x_N}(x,y)=-\frac{2(x_N-y_N)}{\eps_N^2}-\frac{\partial\psi_R}{\partial x_N}(x,y)+
	\frac{d}{ds}\mathrm{d}_\nu(0) = \frac{2y_N}{\eps_N^2}+\nu>0 \; ,$$
    which is a contradiction. Therefore $x_N$ cannot be zero and a similar argument shows that $y_N>0$ as well,
    hence both $x$ and $y$ are inside $\dom$.
\smallskip

	\noindent\textbf{Step 2 --} \textit{Writing the viscosity inequalities and sending $\delta$ to zero.}\\
    We introduce a (small) fixed parameter $0<\delta<1$ which is the parameter appearing in Definition~\ref{defin}
    in order to give sense to different terms in the equation.
    We write the definition of viscosity sub and super solutions, using the test-function
	in the ball $B_\delta$ for a $\delta<\rho:=\min(x_N,y_N,1)$, and the functions $u$ and $v$ outside this ball.
	Since $u$ is a viscosity subsolution and the function $u(\cdot)-v(y)-\varphi(\cdot,y)$ reaches a maximum
	at $x$, then we have the viscosity subsolution condition that we write as follows,
	thanks to Lemma \ref{cor:compensator}:
	$$\begin{aligned}
	&u(x)-\int_{|z|<\delta} [\varphi(x+\eta(x,z),y)-\varphi(x,y)-D_x\varphi(x,y)\cdot\eta(x,z)]\,d\mu(z)\\
	&-\mathrm{P.V.}\,\int_{|z|<\delta}D_x\varphi(x,y)\eta(x,z)\,d\mu(z)-
	\int_{|z|\geq\delta} [ u(x+\eta(x,z))-u(x)] \,d\mu(z)\leq f(x)\,.
	\end{aligned}$$

	For simplicity of notations, we leave out the $\mathrm{P.V.}$ notation since
	the integral can be expressed
	as converging integrals and we use the notation $P(x,z):=x+\eta(x,z)$.
	Next we use Lemma \ref{lem:mu.sharp}-$(i)$ and the similar super solution condition on $v$ to get

	$$\begin{aligned}
	&-\int_{|z|<\delta} [\varphi(P(x,z),y)-\varphi(x,y)-D_x\varphi(x,y)\cdot\eta(x,z)]\,d\mu(z)\\
	&-D_x\varphi(x,y)\cdot\int_{|z|<\delta}\eta(x,z)\,d\mu_\#(z)\\
	&-\int_{|z|<\delta}[\varphi(x,P(y,z))-\varphi(x,y)+D_y\varphi(x,y)\cdot\eta(y,z)]\, d\mu(z)\\
	&+D_y\varphi(x,y)\cdot\int_{|z|<\delta}\eta(y,z)\,d\mu_\#(z)\\
	&-\int_{|z|\geq\delta}[u(P(x,z))-v(P(y,z))-u(x)+v(y)]\,d\mu(z)\\
	&+ u(x)-v(y)\leq f(x)-f(y) \,.
	\end{aligned}$$
    In order to pass to the limit as $\delta\to0$ to get rid of the test-function $\varphi$, we use
    Lemma~\ref{cor:compensator} for all the terms which are smooth functions: the integrals
    over $B(0,\delta)$ all vanish as $\delta\to0$ and we are left with
    limit of the integral over
    $\{|z|>\delta\}$.
	To this end, we split this integral into two integrals, one over $\{|z|\geq1\}$
	(which is independent of $\delta$ of course) and the other over $\{\delta\leq|z|<1\}$ that we have
	to deal with.

	Using the definition of the	maximum point for $\Phi$, we have that for any $z$:
	$$u(P(x,z))-v(P(y,z))-\varphi(P(x,z),P(y,z))\leq u(x)-v(y)-\varphi(x,y)\,.$$
	Hence, it follows that
	\begin{align*}
		 &u(P(x,z))-v(P(y,z))-(u(x)-v(y))  \\[2mm]
		&\leq\frac{|P(x,z)_N-P(y,z)_N|^2}{\eps_N^2}-\frac{|x_N-y_N|^2}{\eps_N^2}
		+\frac{|P(x,z)'-P(y,z)'|^2}{\eps'^2}-\frac{|x'-y'|^2}{\eps'^2}\\
                &\quad + \psi_R(P(x,z),P(y,z))-\psi_R(x,y)\\
&\quad		 -\mathrm{d}_\nu(P(x,z)_N)+\mathrm{d}_\nu(x_N)-\mathrm{d}_\nu(P(y,z)_N)+\mathrm{d}_\nu(y_N)\,,
	\end{align*}
	and we put this inequality into the integral over $\{\delta\leq|z|<1\}$ which gives rise to
	several terms denoted by (with obvious notation):
	$$
	    \int_{\delta\leq|z|<1}\big\{u(P(x,z))-v(P(y,z))-u(x)+v(y)\big\}\,d\mu(z)\leq
	    T_{\eps_N}^\delta+T_{\eps'}^\delta+T_{\psi_R}^\delta+T_{\mathrm{d}_\nu}^\delta\,.
	$$
	As for the $\eps_N$-terms, we get rid of them by $(\mathrm{H}_\eta^5)$ which implies
	$T_{\eps_N}^\delta\leq0$.
	Then for the $\eps'$-terms we write
	$$\begin{aligned}
	T_{\eps'}^\delta =&\int_{\delta\leq|z|<1}\Big(\frac{|P(x,z)'-P(y,z)'|^2}{\eps'^2}-\frac{|x'-y'|^2}{\eps'^2}\Big)d\mu\\
	\leq &\frac{1}{\eps'^2}\int_{\delta\leq|z|<1} |\eta(x,z)'-\eta(y,z)'|^2\, d\mu(z)\\
		&+\frac{2}{\eps'^2}\int_{\delta\leq|z|<1}(x'-y')\cdot(\eta(x,z)'-\eta(y,z)')\,d\mu(z)\,.
	\end{aligned}$$
	For the first term of $T_{\eps'}^\delta$, we use the
        domination of the integrand by $c|z|^2$ to pass to the limit
	as $\delta\to0$. For the second one, we use
        Lemma~\ref{lem:mu.sharp}-$(iii)$ which allows to
	wipe out the symmetric $\mu_*$-contribution, so that we get in the limit
	\begin{align*}
	\limsup_{\delta\to0}T_{\eps'}^\delta\leq&\ \frac{1}{\eps'^2}\int_{0<|z|<1} |\eta(x,z)'-\eta(y,z)'|^2\, d\mu(z)\\
	&+\frac{2}{\eps'^2}\int_{0<|z|<1}(x'-y')\cdot(\eta(x,z)'-\eta(y,z)')\,d\mu_\#(z)\,.
	\end{align*}

	We concentrate now on the penalisation terms which are given by integrals of smooth functions.
	Note first that we are in the case when $0<x_N,y_N<1$, so that
	the $\psi(x_N/R)$ and $\psi(y_N/R)$-terms vanish (we assumed that $R\gg1$).
	Hence, using Lemma \ref{cor:compensator} we get as
        $\delta\to0$ the following two contributions:
	\begin{align*}
	\lim_{\delta\to0}T_{\mathrm{d}_\nu}^\delta=&-\tilde I_1[\mathrm{d}_\nu](x) -\,\frac{d}{ds}\mathrm{d}_\nu(x_N)\,
	\mathrm{P.V.}\,\int_{0<|z|<1}\eta(x,z)_N\,d\mu(z)+(\dots)(y),\\
	\lim_{\delta\to0}T_{\psi_R}^\delta\leq&-\tilde I_1[\tilde\psi_R](x) -D\tilde\psi_R(x)'\cdot
	\mathrm{P.V.}\,\int_{0<|z|<1}\eta(x,z)_N\,d\mu(z)+(\dots)(y).
	\end{align*}
	where $\tilde\psi_R(x):=\psi(|x'|/R)$ and the $(\dots)(y)$
        notation stands for the same terms but calculated at $y$
        instead of $x$.  Now, note that
        $\frac{d}{ds}\mathrm{d}(x_N)>0$ and use
        Lemma~\ref{lem:mu.sharp}-$(iii)$ (with $r=x_N>0$).  This gives
        that in the principal value for $T_{\mathrm{d}_\nu}$, the
        $\mu_*$-term which is multiplied by $(-\nu)$ has a nonpositive
        contribution. So we find that
    $$
	\lim_{\delta\to0}T_\mathrm{d_\nu}^\delta\leq -\nu\,\tilde I_1[\mathrm{d}](x)-\nu\,\frac{d}{ds}\mathrm{d}(x_N)\,
	\int_{0<|z|<1}\eta(x,z)_Nd\mu_\#(z)+(\dots)(y)=o_\nu(1)\,.
	$$
    As for the $T_{\psi_R}$-term, this time
	we use Lemma \ref{lem:mu.sharp}-$(ii)$, which implies that the symmetric $\mu_*$-part of
	the principal value vanishes:
	$$\begin{aligned}
	\lim_{\delta\to0}T_{\psi_R}^\delta&= -\tilde I_1[\tilde\psi_R](x)-\frac{1}{R}\left[\frac{d\psi}{ds}(|x'|/R)\right]'\cdot
	\int_{0<|z|<1}\eta(x,z)'\,d\mu_\#(z)+(\dots)(y)\\
	&\leq C(\mu)\Big(\frac{1}{R^2}\|\psi\|_{C^2}+\frac{1}{R}\|\psi\|_{C^1}\Big)=o_R(1)\,.
	\end{aligned}
    $$

	Thus, the parameters $\eps',\eps_N,\nu,R>0$ are still fixed for the moment and after sending $\delta\to0$
	we have obtained:
	$$\begin{aligned}
	u(x)-v(y)\leq &\ f(x)-f(y)+o_\nu(1)+o_R(1)\\
	&+\frac{1}{\eps'^2}\int_{|z|<1} |\eta(x,z)'-\eta(y,z)'|^2\, d\mu(z)\\
	&+\frac{2}{\eps'^2}\int_{|z|<1}(x'-y')\cdot(\eta(x,z)'-\eta(y,z)')\,d\mu_\#(z)\\
	&+\int_{|z|\geq1}\big\{u(P(x,z))-v(P(y,z))-u(x)+v(y)\big\}\,d\mu(z)\\
	&=f(x)-f(y)+o_\nu(1)+o_R(1)+\mathrm{Int_1}+\mathrm{Int_2}+\mathrm{Int_3}\,.
	\end{aligned}$$
\smallskip

	\noindent\textbf{Step 3 -- } \textit{Sending the parameters to
          their limits }\\
        We let first $\eps_N\to0$, the other parameters remaining
        fixed for the moment and we recall that $|x_N-y_N|=O(\eps_N)$.
        Moreover, as long as $R$ is fixed, the points $x,y$ remain in
        a compact subset of $\domb$; therefore we can assume without
        loss of generality that $x,y$ are converging to points (still
        denoted by $x,y$) such that $x_N=y_N$.

    Combining $(\mathrm{H}_\eta^3)$ and $(\mathrm{H}_\eta^4)$, we
    obtain in particular that
$$\lim_{\eps_N\ra0}|\eta(x,z)'-\eta(y,z)'|\leq
C|z||x'-y'|\quad\text{for $\mu$-a.e. }z\,.$$
Then, $(\mathrm{H}_\eta^1)$ and the integrability
    condition on $\mu_\#$ justify that we can use dominated
    convergence in $\mathrm{Int_2}$. The argument is similar for $\mathrm{Int_1}$, using the domination
	$$|\eta(x,z)'-\eta(y,z)'|^2\leq (2c_\eta)^2|z|^2\,.$$
	So we find that  $\lim_{\eps_N\ra0}\mathrm{Int_2}=0$ while
$$\lim_{\eps_N\ra0}\mathrm{Int_1}\leq C\frac{|x'-y'|^2}{\eps'^2}\int_{|z|<1}|z|^2d\mu(z)=o_{\eps'}(1).
 $$

The $o_R(1)$ and $o_\nu(1)$ terms are uniform with respect to the other
parameters, so there is no problem to send $\eps_N,\eps'\to0$
here. Next, since $|x-y|\ra0$ as $\eps_N,\eps'\to0$
here, we may assume that $x,y\ra \bar x\in\dom$ by considering a
subsequence is necessary. By continuity of $f$, it then follows that
$\big(f(x)-f(y)\big)\ra 0$.

    We then pass also to the limit as $\nu\to0$ and get:
    \begin{equation}\label{ineq:ultimate.step}
    u(\bar{x})-v(\bar{x})\leq
    \limsup_{\nu\to0}\,\limsup_{\eps'\to0}\,\limsup_{\eps_N\to0}\,\left[\mathrm{Int_3}\right]
    +o_R(1)\,.
    \end{equation}
    Passage to the limit in the $\mathrm{Int_3}$ term is possible because the
	domain of integration does not meet the singularity of the
        integral: we need only use the u.s.c.
	and l.s.c. properties of $u$ and $v$, together with Fatou's Lemma (because the integrand is bounded
	and $\mu$ is finite on $\{|z|\geq1\}$).
	After passing to the limit in $\eps_N,\eps'$ and $\nu$,
	we have by definition
	$$\lim_{\nu\to0}\lim_{\eps'\to0}\lim_{\eps_N\to0}\big(u(x)-v(y)\big)=M+o_R(1)$$ so that
\begin{align*}
&\limsup_{\nu\to0}\,\limsup_{\eps'\to0}\,\limsup_{\eps_N\to0}\,\left[\mathrm{Int_3}\right]\\
&\leq
	\int_{|z|\geq1}\big\{u(P(\bar x ,z))-v(P(\bar x,z))-\big(M+o_R(1)\big)\big\}\,d\mu(z)\,.
\end{align*}
	Now since and $u(P(\bar x ,z))-v(P(\bar x,z))\leq \sup_{\domb}(u-v)=M$,
$$\limsup_{\nu\to0}\,\limsup_{\eps'\to0}\,\limsup_{\eps_N\to0}\,[\mathrm{Int}_3] \leq \int_{|z|\geq1}\, o_R(1) \,d\mu(z) = o_R(1).$$
When  $R\to\infty$ in \eqref{ineq:ultimate.step}, we get
$M\leq0$ and the proof is complete.
\end{proof}


\section{Comparison in the censored case I.}
\label{Sec:alp1}

In this section we give comparison and well-posedness results for the initial
value problem \eqref{E} in the censored case (under assumption
$(\mathrm{H}_\eta^6)$) when the measure $\mu$ is not too singular:
\begin{itemize}
\item[$(\mathrm{H}_\mu')$] The measure $\mu$ is a nonnegative Radon
  measure satisfying
	\begin{align*}
	(i)\ \ 	 \int_{\R^N}|z|\wedge1\,d\mu<\infty\quad\hbox{and} \quad (ii)\ \ 	\int_{\{z_N=a\}}\,d\mu = 0\quad\hbox{for any  }a<0\; .
	\end{align*}
\end{itemize}
In addition, we assume the existence of a ``blow-up supersolution''
\begin{itemize}
\item[$(U) $] There exists $R_0>0$ such that, for any $R>R_0$, there is a
positive function $\U_R\in C^2(\dom)$ such that
\begin{align*}
-I[\U_R](x)&\geq-K_R\quad\text{in}\quad \{x:0<x_N\leq R\},
\\
\intertext{for some $K_R\geq0$, and}
\U_R(x) &\geq
\frac1{\omega_R(x_N)}\quad \text{in}\quad \dom,
\end{align*}
for some function $\omega_R$ which is nonnegative, continuous, strictly increasing
in a neighbourhood of $0$, and  satisfies $\omega(0)=0$.
\end{itemize}

\begin{remark}
See Appendix \ref{App:Blowup} for a discussion of this
assumption. E.g. in Remark \ref{RemA1} we
prove that (U) holds if
$$\mu=\bar\mu+\sum_{i=1}^Mc_i\delta_{x^i}
$$
where $c_i\in\R$, $\delta_{x^i}$ are delta measures supported at $\{x^i\}$ for
$x^i_N>0$, and
$$\frac{d\bar\mu}{dz}=\frac{g(z)}{|z|^{N+\alp}}\qquad\text{where}\quad
\alp\in(0,1),\ 0\leq g\in
L^\infty(\R),\ \lim_{z\ra
  0}g(z)=g(0)>0.$$
This class of measures include the L\'evy measures of the stable,
tempered stable, and self-decomposable L\'evy processes. Much more
general examples are presented in Appendix \ref{App:Blowup}.
\end{remark}

\begin{theorem}
\label{thm:censor1}
Assume $(\mathrm{H}_\mu')$, $(\mathrm{H}_f)$, $(\mathrm{H}_\eta^6)$ and
(U) hold. Let  $u$ be a bounded usc subsolution of \eqref{E} and
$v$ be a bounded lsc supersolution of \eqref{E}. Then $u\leq
v$ in $\domb$.
\end{theorem}

As in the previous section, we immediatly get a well-posedness result
for \eqref{E} by Theorems \ref{thm:censor1} and \ref{limit}.

\begin{corollary}
Under the assumptions of Theorem \ref{thm:censor1}, there exists a unique
viscosity solution of \eqref{E} depending continuously on $f$.
\end{corollary}

\begin{proof}[Proof of Theorem \ref{thm:censor1}] We argue by contradiction assuming that $M:=\sup_{\overline{\Omega}}\,(u-v)>0$.
Take $R>R_0$ and $0<\kappa \ll 1$. Using $0<\eps \ll 1$, we double the variables by introducing the quantities
\begin{align*}
&\phi(x,y)=
\frac{|x-y|^2}{\eps^2}+\kappa[\U_R(x)+\U_R(y)]+\psi_R(x)+\psi_R(y),\\
&\Phi(x,y)=u(x)-v(y)-\phi(x,y),
\end{align*}
where $\U_R$ is given by (U) and
$\psi_R(x)=2(\|u\|_\infty+\|v\|_\infty)\psi(\frac{|x|}R)$ for an increasing
function $\psi(s)\in
C^\infty(0,\infty)$ which is $0$ in $(0,\frac12)$ and $1$ in $(1,\infty)$.

For any $R, \kappa$ and $\eps$, the function $\Phi$ achieves its maximum at $(\bar x,\bar y)=(\bar
x_{R,\kappa,\eps},\bar y_{R,\kappa,\eps})$ and, by
the definition of $\U_R$ and $\psi_R$, we have
\begin{equation}\label{propbxby}
\bar x_N,\bar y_N \geq \delta_0
=\omega_R^{-1}\Big(\frac\kappa{2(\|u\|_\infty+\|v\|_\infty)}\Big) \qquad\text{and}\qquad|\bar
x|,|\bar y|\leq R.
\end{equation}
These estimates will hold in most of the proof since we are going to
keep  $R$ and $\kappa$ fixed untill the end, sending  $\eps\ra0$ first.
A standard argument also shows that
$$\frac{|\bx-\by|^2}{\eps^2}\ra0\quad\text{as}\quad \eps\ra0.$$
By the estimates on $\bx,\by$ and extracting a
subsequence if necessary, we can assume without
loss of generality that $\bx,\by \to X$, $u(\bx)\to u(X)$, and
$v(\by) \to v(X)$ where $X$ is a maximum point of
$\Phi(x,x)=u(x)-v(x)-\phi(x,x)$.
Finally, when we first send $\kappa \to 0$ and then $R\to +\infty$, we have
$$u(X)-v(X) \to M\quad\hbox{and}\quad \kappa\U_R(X)+\psi_R(X) \to 0\; .$$

Now we write down the viscosity inequalities. Since $u-\phi(\cdot,\bar y)$ has a global maximum at
$\bar x$ and $v-(-\phi(\bar x,\cdot))$ has a global minimum at
$\bar y$, we have that
\begin{align*}
u(\bx)-I^\delta[u](\bx)-I_\delta[\phi(\cdot,\by)](\bx)&\leq f(\bx),\\
v(\by)-I^\delta[v](\by)-I_\delta[-\phi(\bx,\cdot)](\by)&\geq f(\by).
\end{align*}
With this in mind we see that
\begin{align}
&M + o (1)=u(\bx)-v(\by) - \phi(\bx,\by)\nonumber\\
&\leq I^\delta[u](\bx)-I^\delta[v](\by)+I_\delta[\phi(\cdot,\by)](\bx)-I_\delta[\phi(\bx,\cdot)](\by)+f(\bx)- f(\by).\label{Mineq}
\end{align}
In this inequality, we aim at first sending $\delta \to 0$ in order to
get rid of the $\eps$-depending $I_\delta[\phi]$-terms.
In fact $I_\delta[\varphi]\ra0$ as
$\delta\ra0$ by the Dominated Convergence Theorem since
$|\eta(x,z)|\leq c_\eta|z|$,
and hence for any $C^1$-function $\varphi$,
\begin{align*}
&\int_{\R^N}|\varphi(x+\eta(x,z))-\varphi(x)|\
1_{|z|<\delta}d\mu(z) \leq
c_\eta\|D\varphi\|_{L^\infty(B_{c_\eta\delta})}\int_{\R^N} 1_{|z|<\delta}\ |z|\ d\mu(z).
\end{align*}

Next we consider the $I^\delta$-terms. We restrict ourselves to a
subsequence such that $\bx_N\geq\by_N$ (if $\bx_N\leq\by_N$ the
argument is similar). Then
\begin{align*}
I^\delta[u](\bx)-I^\delta[v](\by)&=\int_{-\bx_N<z_N<-\by_N}[u(\bx+z)-u(\bx)]\
1_{|z|>\delta}d\mu(z)\\
&\quad + \int_{-\by_N<z_N}[u(\bx+z)-v(\by+z)-(u(\bx)-v(\by))]\
1_{|z|>\delta}d\mu(z)\\
&=:I_1+I_2.
\end{align*}
For $I_1$, we have
$$|I_1|\leq 2\|u\|_{\infty}\int_{|z|>\delta}
1_{\{-\bar x_N<z_N<-\bar y_N\}}(z) \ d\mu(z)\; .$$
Keeping $\kappa$ and $R$ fixed and recalling (\ref{propbxby}), we see that this integral is independent of $\delta$ as soon as $\delta <\delta_0$. Furthermore, because of $(\mathrm{H}_\mu')\  (ii)$, the Dominated Convergence Theorem implies that
$$ I_1 \ra 0\quad\text{as}\quad\eps\ra0$$
since $|\bx-\by|\ra0$ as $\eps\ra0$.

For $I_2$, we use the maximum point property for $\bx,\by$,
$$ \big(u(\bx+z)-v(\by+z)\big)-\big(u(\bx)-v(\by)\big) \leq \phi(\bx+z,\by+z)-\phi(\bx,\by)\; ,$$
which after cancellation of the $\eps$-terms leads to
$$ I_2 \leq  \kappa\Big( I^\delta[\U_R](\bx)+I^\delta[\U_R](\by)\Big) + \Big(I^\delta[\psi_R](\bx)+I^\delta[\psi_R](\by)\Big)\; .$$
Recalling again (\ref{propbxby}) and using the regularity of $\U_R$ and $\phi$, we can send $\delta\ra0$ and obtain
$$\limsup_\delta \,I_2 \leq \kappa\Big( I[\U_R](\bx)+I[\U_R](\by)\Big) + \Big(I[\psi_R](\bx)+I[\psi_R](\by)\Big)\; ,$$
where each term on the right-hand side have a sense.

Consider equation (\ref{Mineq}) again. Using all the previous
estimates, we can send $\delta\ra0$ first and obtain using (U) for
the $\U_R$-terms that
$$
M + o (1) \leq 2K_R\kappa + (I[\psi_R](\bx)+I[\psi_R](\by))+ (f(\bx)-f(\by))\; .$$
In this inequality, we can first send $\eps \to 0$, keeping $R$ and
$\kappa$ fixed. Then $f(\bx)-f(\by)\ra 0$ as $\eps\ra0$
since $f$ is uniformly continuous in $\overline{B}_R$, and we find that
$$
M + o (1) \leq 2K_R\kappa + 2I[\psi_R](X)\; .$$
We conclude by first sending $\kappa \to 0$ and then $R\to +\infty$.
\end{proof}

\section{Comparison results in the censored case II.}
\label{sec:sens2}

In this section we give comparison and well-posedness results for the
initial value problem \eqref{E} in the censored case (under assumption
$(\mathrm{H}_\eta^6)$) when the measure $\mu$ is very
singular

\smallskip
\noindent $(\mathrm{H}_\mu'')$ Hypothesis $(\mathrm{H}_\mu)$ holds with
$$\mu_*(dz)=\dfrac{dz}{|z|^{N+\alp}}\;,\;\int_{\R^N}(1\wedge|z|^\tibeta
)\,\mu_\#(dz)<\infty \; , \; \int_{\{z_N=a\}}\mu_\#(dz) = 0\;\hbox{for any  }a<0\; ,
$$
for $\alp\in(1,2)$ and $\tibeta := \alp -1$.
\smallskip

This assumption is much more restrictive than $(\mathrm{H}_\mu)$, and
the results of this section are not completely satisfactory. We had
lot of difficulties to obtain comparison results  because on one
hand, it is not possible to get rid of the boundary and the boundary
condition in such a general way as we did in the less singular case
I.  On the other hand a lot of technical difficulties come from the
the way the $x$-depending domain of integration in  $I$ interferes
with the singularity of the measure and the boundary.

Our first result is the following

\begin{theorem}\label{thm:comp.censored.reg} Assume
  $(\mathrm{H}_f)$, $(\mathrm{H}_\eta^6)$, and $(\mathrm{H}_\mu'')$
  hold.

\noindent(a) Let $u$ and $v$ be respectively a bounded usc subsolution
and a bounded lsc supersolution of
\begin{equation}\label{eqn-in}
 w(x)-I[w](x)=f(x) \quad\text{in}\quad \dom\; ,
\end{equation}
and let us also denote by $u$ and $v$ respectively their usc or lsc
extensions to $\domb$\footnote{For any $x' \in \R^{N-1}$,
  $\displaystyle u(x',0) :=\limsup_{(y',y_N) \to (x',0)}u(y',y_N)$ and
  $\displaystyle v(x',0) :=\liminf_{(y',y_N) \to (x',0)}v(y',y_N)$}.
If there exists $C>0$ and $\beta>\tibeta$ such that
\begin{equation}\label{cond-0}
u(x',x_N) \geq u(x',0) -Cx_N^\beta \quad\hbox{and}\quad v (x',x_N) \leq v(x',0) +Cx_N^\beta
\end{equation}
then $u$ and  $v$ are respectively a bounded usc subsolution and a
bounded lsc supersolution of \eqref{E}.

\noindent(b) If $u$ and $v$ are respectively a bounded usc subsolution
and a bounded lsc supersolution of \eqref{E} satisfying
\eqref{cond-0}, then
$$u\leq v\quad\text{in}\quad\domb.$$
In particular, there exists at most one solution
of \eqref{E} in $C^{0,\beta}(\dom)$ for $\beta>\tibeta$.
\end{theorem}

Several comments have to be made on the different statements in
Theorem~\ref{thm:comp.censored.reg}. Part (a) means that, for sub and
supersolutions having a suitable regularity at the boundary, the
Neumann boundary condition is already encoded in the equation
inside. This might be expected from the proof of Theorem~\ref{limit}
or from the intuition coming from the censored process. But the result
is not true in general since we need anyway \eqref{cond-0} to prove
it.


Unfortunately part (b) does not provide the full comparison result
for semi-continuous solutions, and  we do not know if this result is
optimal or not. Of course, in view of Theorem~\ref{thm:comp.censored.reg}
(b), it is clear that we need a companion existence result providing the
existence of solutions satisfying \eqref{cond-0} or belonging to
$C^{0,\beta}(\dom)$ for $\beta>\tibeta$. We address this
question after the proof of Theorem~\ref{thm:comp.censored.reg}.

\begin{proof}
We prove (a) only in the subsolution case since the
supersolution case is analogous. Let $\phi$ be a smooth function
which is bounded and has bounded first and second-order derivatives
and assume that $u-\phi$ has a maximum point $(x',0)\in\del\dom$
 in $B((x',0),c_\eta\delta)\cap\domb$. We may assume
that the maximum is strict and global without any loss of generality.

We set $\theta(t) = t^\tibeta \wedge 1$  for $t\geq 0$ and, for $0<\kappa \ll 1$, we consider
the function $u(x)-\phi(x) +\kappa \theta(x_N)$. By standard
arguments, using the properties of $\phi$, this function
achieves a global maximum at a point nearby $(x',0)$, and we claim that
this point cannot be on $\del\dom=\{x:\ x_N = 0\}$. Indeed, otherwise
it would have to be $(x',0)$, the strict global maximum point
of $u-\phi$ on $\del\dom$. But then by \eqref{cond-0},
$$ u(x',0) - \phi(x',0) \geq u(x)-\phi (x)+\kappa \theta(x_N) \geq  u(x',0) - \phi(x',0) - 2Cx_N^\beta +\kappa \theta(x_N),$$
and we have a contradiction since $\beta > \tibeta$ and hence $-
2Cx_N^\beta +\kappa \theta(x_N)>0$ for $x_N$ small enough.

Therefore the function $x\mapsto u(x)-\phi (x)+\kappa \theta(x_N)$ has a maximum point at $x_\kappa$ with $(x_\kappa)_N >0$.
We may write the viscosity inequality at $x_\kappa$ as
$$u(x_\kappa)-\tilde I_\delta[\phi](x_\kappa)-\gamma(x_\kappa)\cdot
D\phi(x_\kappa) + \kappa I_\delta[ \theta] (x_\kappa)-I^\delta[u](x_\kappa)\leq f(x_\kappa),$$
for (say) $0<\delta <1$, where $\gamma(x_\kappa)=P.V.\int_{|z|<\delta}\eta(x_\kappa,z)\mu(dz)$.

We first consider the term $\kappa I_\delta[ \theta] (x_\kappa)$. On one hand, the $\mu_\#$-part
is $O(\kappa)$ since $\theta$ is in $C^{0,\tibeta}$ and
$(\mathrm{H}_\mu'')$ holds. On the other hand, the singular part (the $\mu_*$ part) is nothing but
$$\kappa\  \mathrm{P.V.}\int_{|z|\leq \delta \atop x_N + z_N \geq 0 }\theta(x_N+z_N)-\theta(x_N) \
\frac{dz}{|z|^{N+\alp}}\; ,$$
where we have dropped the subscript $\kappa$ to simplify the notation.
Since $\delta <1$ and $x_N \to 0$ as $\kappa\ra0$, we may assume that
$0 \leq x_N+z_N <1$ for $|z|\leq \delta$, and hence that the principal
value reduces to
$$\kappa\  \mathrm{P.V.}\int_{|z|\leq \delta \atop x_N + z_N \geq 0 }|x_N+z_N|^\tibeta-|x_N|^\tibeta \
\frac{dz}{|z|^{N+\alp}}\; .$$
By the computations of Lemma \ref{alp-lem} in the Appendix,
$$ - \mathrm{P.V.}\int_{x_N + z_N \geq 0}|x_N+z_N|^\tibeta-|x_N|^\tibeta \
\frac{dz}{|z|^{N+\alp}} = 0$$
for $x_N >0$. Writing
$$\kappa\  \mathrm{P.V.}\int_{|z|\leq \delta \atop x_N + z_N \geq 0 } (\cdots)
 = \kappa\  \mathrm{P.V.}\int_{x_N + z_N \geq 0 }(\cdots) \
 -\kappa\ \int_{|z| > \delta \atop x_N + z_N \geq 0 }(\cdots) \
\; ,$$
we conclude that for fixed $\delta$,
$$\kappa\  \mathrm{P.V.}\int_{|z|\leq \delta \atop x_N + z_N \geq 0 }\theta(x_N+z_N)-\theta(x_N) \
\frac{dz}{|z|^{N+\alp}} = O(\kappa)\; .$$

Finally, the $u$, $\tilde
I_\delta$, and $I^\delta$ terms are uniformly bounded in $\kappa$ while
$\gamma(x_\kappa)\ra\infty$ since $(x_\kappa)_N\ra0$. We divide
the above inequality by $|\gamma(x_\kappa)|$ and send
$\kappa\ra0$. As in the proof of Theorem~\ref{limit} -- the second part, when
$x\in\partial\Omega$ and $c=1$ -- the result is that all terms vanish
except the $\gamma$-term and we are left with the boundary condition
$$\frac{\partial\phi}{\partial\mathbf{n}}(x)\leq0\,.$$

Now we prove part (b). By linearity of the problem and part (a),
the function $w=u-v$ is a subsolution of \eqref{E} with $f \equiv
0$, and we are done if we can prove that $w\leq0$. To prove this, we
consider the function
$$ \chi_{R,\nu}(x) := \psi(|x_N|/R)+\psi(|x'|/R) - \nu\mathrm{d} (x_N)\; ,$$
where $\psi$ and $\mathrm{d}$ are defined as in the proof of
Theorem~\ref{fleas}, replacing, in the case of $\psi$,
$2(\|u\|_\infty+\|v\|_\infty + 1)$ by $2\|w\|_\infty + 1$. The
function $\chi_{R,\nu}$ is smooth and easy computations show that
$\chi_{R,\nu}$ is a supersolution of \eqref{E} with an $f\geq\varpi(R,\nu)$
where $\varpi(R,\nu)$ converges uniformly to $0$ as $R\to \infty$ and
$\nu \to 0$. At the boundary $\del\dom$,
$$ - \frac{\partial \chi_{R,\nu}}{\partial x_N} =0+\nu\cdot 1>0
.$$
Because of the behavior of $\chi_{R,\nu}$ at infinity, the function
$w-\chi_{R,\nu}$ achieves its maximum at some point $x$,
and because of the behaviour of $\chi_{R,\nu}$ at the boundary,
$x_N>0$. Writing the viscosity subsolution inequality then yields that
$$w(x) -\chi_{R,\nu}(x)\leq
-\chi_{R,\nu}(x)+I[\chi_{R,\nu}](x)+I^\delta[u-\chi_{R,\nu}](x)\leq
-\varpi(R,\nu)+0,$$
where we have used that $I^\delta[\psi](x)\leq 0$ at any maximum point
$x$ of $\psi$. Hence, for any $y \in \domb$,
$$  w(y)-\chi_{R,\nu}(y) \leq -\varpi(R,\nu),$$
and part (b) follows from sending $R\to \infty$ and then $\nu \to 0$.
\end{proof}

Now we turn to the existence of H\"older continuous solutions and we begin with a result in $1$-d.

\begin{theorem}\label{exist1}  Assume $N=1$ and that
  $(\mathrm{H}_f)$, $(\mathrm{H}_\eta^6)$, and $(\mathrm{H}_\mu'')$
  hold.

\noindent (a) Any bounded, uniformly continuous solution of \eqref{E} is in
$C^{0,\beta}(\domb)$ for some $\beta >\tibeta$.

\noindent (b) There exists a solution of \eqref{E} in
$C^{0,\beta}(\domb)$ for some $\beta >\tibeta$.
\end{theorem}

\begin{proof} (a) To prove the H\"older reglarity we consider
\begin{equation}\label{themax}
M = \sup_{[0,+\infty)\times [0,+\infty)}\, (u(x)-u(y) - C|x-y|^\beta) \; ,
\end{equation}
and argue by contradiction assuming that $M>0$. The aim is to show
that this is impossible for $C>0$ large enough.  A rigorous proof
would consists in introducing localization terms like the $\psi$-terms
in the proof of Theorem~\ref{fleas} and $\mathrm{d}_\nu$-terms in
order to take care of the Neumann boundary condition. We drop these
terms for the sake of simplicity in order to emphasize the main idea
and not loose the reader in technicalities.

Therefore we assume that the above supremum is achieved at $(x,y)$
with $x,y>0$.
Since $M>0$ we have $x \neq y$, and we assume below $x<y$. The other
case can be treated analogously.
 To simplify the notation, we introduce the
function $\phi(z):=C|x-y+z|^\beta$. Note that this function is
concave in the intervals $(-\infty, y-x)$ and $(y-x,+\infty)$, and
that it is smooth in $(-\delta,\delta )$
for $\delta\leq y-x$ so that it can be used as a test
function.
 By the maximum point property for $(x,y)$,
$$ u(x+z_1)-u(y+z_2) - C|x-y+(z_1-z_2)|^\beta \leq u(x)-u(y) - C|x-y|^\beta\; ,$$
for $z_1\geq -x$ and $z_2 > -y$, and hence
\begin{align}
 u(x+z)-u(y+z) - [ u(x)-u(y)]&\leq0 \qquad\text{for}\qquad z\geq -x\ (> -y),\label{m1}\\
u(x+z)-u(x) & \leq  [\phi(z) - \phi(0)] \quad\text{for}\quad z\geq -x,\label{m2}\\
u(y+z)-u(y) &\geq - [\phi(-z) - \phi(0)] \quad\text{for}\quad z\geq
-y.\label{m3}
\end{align}

Using the definition of viscosity solution and the symmetry of the measure $\mu_*$, for $\delta, \delta' >0$
small enough, we have the inequalities
$$-(I_\delta[\phi]+I^\delta[u])(x)+u(x)\le f(x)\quad\text{and}\quad -(I_{\delta'}[\phi]+I^{\delta'}[u])(y)+u(y)\ge f(y),$$
which reduce here to
\begin{align}
   &-\int_{-x}^{-\delta} (u(x+z)-u(x))d\mu(z)-\int_{-\delta}^{\delta}
   [\phi(z) - \phi(0)-\phi'(0)z]d\mu(z) \label{vi1}\\
   &\quad -\int_{-\delta}^{\delta} \phi'(0)z d\mu(z)-\int_{\delta}^{+\infty} (u(x+z)-u(x))d\mu(z) + u(x) \leq f(x)\; ,\nonumber
\end{align}
\begin{align}
&-\int_{-y}^{-\delta'} (u(y+z)-u(y))d\mu(z)  +
\int_{-\delta'}^{\delta'} [\phi(-z) - \phi(0)+
\phi'(0)z]d\mu(z)\label{vi2}\\
&\quad
-\int_{-\delta'}^{\delta'} \phi'(0)z d\mu(z) -\int_{\delta'}^{+\infty} (u(y+z)-u(y))d\mu(z) + u(y) \geq f(y)\; .\nonumber
\end{align}

In the proof below we will subtract these inequalities and the main
difficulty of the proof will come from the term
$$J:=-\int_{-y}^{-x} (u(y+z)-u(y))d\mu(z)\,$$
which is {\em not} a difference of terms from  \eqref{vi1} and \eqref{vi2}.
Indeed the domain of integration $z\in(-y,-x)$ appears in inequality
 \eqref{vi2} but not in  \eqref{vi1}. Because of the singularity of
 $\mu$, if $x$ is close to $0$ it is not obvious how to get an
 estimate for $J$ which is independent of $C$, or how to control this
 ``bad term'' by a good
term. Therefore we have problems with this term if $x\to 0$ when $C\to
+\infty$. For the $\mu_\#$-part of $J$ there is no problem, we can use
\eqref{m3} to see that
\begin{equation*}\label{mudiese}
-\int_{-y}^{-x} (u(y+z)-u(y))d\mu_\#(z) \leq \int_{-y}^{-x} [\phi(-z) - \phi(0)] d\mu_\#(z) \leq C\int_{-y}^{-x} |z|^\beta d\mu_\#(z)\; ,
\end{equation*}
and we will see later that this term can be controlled since
$|z|^\beta$ is  $\mu_\#$-integrable.
\medskip

\noindent\textsc{First case --}
We first consider the case when $x \leq y-x$, or equivalently, $2x \leq y$. In
this case $J\geq0$ and can be dropped from inequality \eqref{vi2}. To
see this we note that for $-y \leq z \leq -x$,
$$ 2x-y \leq x-y-z \leq x$$
with  $x\leq y-x$ and $2x-y=-(y-x)+x \geq -(y-x)$, and hence by \eqref{m3}
\begin{align}\label{m4}u(y+z)-u(y) \geq - [\phi(-z) - \phi(0)] = |x-y|^\beta -
  |x-y-z|^\beta \geq 0\; .
\end{align}

In this first case, we choose $\delta = x$ and $\delta' =y- x$ and
subtract the viscosity inequalities \eqref{vi1} and \eqref{vi2}. After
some computations using \eqref{m2}, \eqref{m3}, and \eqref{m4}, and
dropping the $J$ term, we are lead to the inequality
$$ \begin{aligned}
&-\int_{-x}^{y-x} [\phi(z) + \phi(-z)- 2\phi(0)]d\mu(z) \\
&-\int_{y-x}^{+\infty} ((u(x+z)-u(y+z))-(u(x)-u(y)))d\mu(z) + u(x)-u(y) \leq f(x)-f(y)\; .
\end{aligned}$$
Some easy computations then shows that the first integral equals
$$ -C(y-x)^{\beta-\alp} \int_{-\frac{x}{y-x}}^{1} [|1+z|^\beta + |1-z|^\beta- 2]\frac{dz}{|z|^{1+\alp}} + O(C)\; ,$$
where the $O(C)$-term comes from the $\mu_\#$ part of the measure since the
integrand can be estimated by $2|z|^\beta$ which is integrable on,
say, $(-1,1)$. The second integral is nonpositive by \eqref{m1} and
can be dropped because of the ``$-$'' in front.

Finally, since $f$ is
bounded and $u(x)-u(y) \geq 0$ (by assumption), we obtain
 \begin{align}
\label{m5}
-C(y-x)^{\beta-\alp} \int_{-\frac{x}{y-x}}^{1}
   [|1+z|^\beta + |1-z|^\beta- 2]\frac{dz}{|z|^{1+\alp}} \leq
   2\|f\|_\infty + O(C)\; .
\end{align}

In order to conclude, we use that
$M=u(x)-u(y)-C|x-y|^\beta >0$ (by assumption) and $\beta\le 1\le
\alpha$ to find that
$$ |x-y|\leq \Big(\frac{2\|u\|_\infty}{C}\Big)^{1/\beta}\quad \text{and}\quad C(y-x)^{\beta-\alp} \geq KC^{\zeta}\; ,$$
where $\zeta:= 1+(\alp-\beta)\beta^{-1} >1$ and $K=(2 ||u||_\infty)^{ \frac {\beta-\alpha} \beta}$. Then we note that
$$-\int_{-\frac{x}{y-x}}^{1} [|1+z|^\beta + |1-z|^\beta- 2]\frac{dz}{|z|^{1+\alp}} \geq
-\int_{0}^{1} [|1+z|^\beta + |1-z|^\beta- 2]\frac{dz}{|z|^{1+\alp}}>0\; ,$$
since $z\mapsto |1+z|^\beta$ is strictly concave on $(-1,1)$. From inequality
\eqref{m5} we then find that
$$\tilde K C^{\zeta} \leq 2\|f\|_\infty + O(C)\;,$$
which cannot hold for $C$ large enough and we have a contradiction in the first case.
\smallskip

\noindent\textsc{Second case --} When $x > y-x$, or equivalently, $2x > y$.
In this case we choose $\delta = \delta' =y- x$, subtract viscosity
inequalities \eqref{vi1} and \eqref{vi2}, and use  \eqref{m3} to see that
$$ \begin{aligned}
-&\int_{-y}^{-x} [\phi(-z) - \phi(0)] d\mu(z)-\int_{-(y-x)}^{y-x} [\phi(z) + \phi(-z)- 2\phi(0)]d\mu(z) \\
&-\int_{y-x}^{+\infty} ((u(x+z)-u(y+z))-(u(x)-u(y)))d\mu(z) + u(x)-u(y) \leq f(x)-f(y)\; .
\end{aligned}$$
Arguing as in the first case, we can drop all $u$-terms and are lead
to an inequality of the form
$$
-C(y-x)^{\beta-\alp}(B(a)+G)\leq 2\|f\|_\infty + O(C)\; ,$$
where
\begin{align*}
B(a)&=\int_{-a-1}^{-a}(|1+z|^\beta-1) \ \frac{dz}{|z|^{N+\alp}},\\
G&=\int_{-1}^1 (|1+z|^\beta + |1-z|^\beta -2)\ \frac{dz}{|z|^{N+\alp}},
\end{align*}
with $a= x/(y-x)>1$.
A technical computation (Corollary~\ref{nec-res} in the appendix) then
shows that
$B(a)+G \leq - \kappa <0$ for some $\beta>\tibeta$ and we can conclude
the argument as in the first case. The proof of (a) is complete.

Note the important estimate, valid in both cases: there
exist $k_1, k_2>0$ such that
\begin{equation}\label{est.Iux.Iuy}
    I[u](x)-I[u](y)\le -k_1C|x-y|^{\beta-\alpha}+k_2(1+C)\,,
\end{equation}
where the $1$ comes from the localization terms. This formal estimate
should be interpreted in the viscosity sense and with the above choice(s) of
test function and parameters $\delta $ and $\delta'$, cf. e.g. \eqref{m5}.
\medskip

\noindent (b) To show the existence of solutions with a suitable regularity
property, we follow the so-called ``Sirtaki method in 4
steps''. We just give a formal sketch the proof
which is an easy adaptation of the above arguments.

We start by building a suitable approximate problem. We
approximate the L\'evy measure $\mu$ by bounded measures
$\mu_n=\mu\,1_{|z|>1/n}$ for $n \geq 1$ and denote the associated
nonlocal term by $I_n$. Then we introduce a truncation of the nonlocal
term and add an additional viscosity term. The results is the approximate equation
$$-\epsilon u_{xx} - T_R(I_n[u]) + u = f \quad \hbox{in  }\dom\; ,$$
where $T_R(s):=\max(\min(s,R),-R)$, $R,\epsilon>0$\,.
\smallskip

\noindent 1. For fixed $\epsilon,n,R$, since the $T_R$-term and the
measure $\mu_n$ are bounded, this equation can
easily be solved by classical viscosity solutions' methods (Perron's
method and comparison result). This provides us with a continuous
solution which is bounded and we even have $||u||_\infty \leq
||f||_\infty$.

Moreover, in order to obtain the $C^{0,\beta}$-regularity and
$C^{0,\beta}$-bounds, we consider \eqref{themax} and follow the
arguments in the first part of this proof. After subtracting the
viscosity sub- and supersolution inequalities, we formally obtain
\begin{align}
-\epsilon \left[ u_{xx}(x)-u_{xx}(y)\right]  - \left[T_R(I_n[u])(x) - T_R(I_n[u])(y)\right] & \nonumber \\
+u(x)-u(y) & \leq
f(x)-f(y) \; .\label{keyineqcob}
\end{align}
For the second-derivatives, we have an analogue estimate to
\eqref{est.Iux.Iuy}, namely there exists $k'_1, k'_2>0$ such that
\begin{equation}\label{est.uxx.uyy}
    u_{xx}(x)-u_{xx}(y)\le -k'_1C|x-y|^{\beta-2}+k'_2\,.
\end{equation}
Note that to give meaning to this formal estimate, we must consider
instead of $u_{xx}$ the sub- and super jets of the theorem of sums,
cf. e.g. \cite{BI}.  Now consider \eqref{keyineqcob} with
fixed $R,\eps>0$. Since the $T_R$-terms are bounded, we can rewrite it as
$$-\epsilon \left[ u_{xx}(x)-u_{xx}(y)\right] \leq 2R+2(||u||_\infty +
||f||_\infty
)\; ,$$
and use
\eqref{est.uxx.uyy} to find that the inequality cannot hold for $C$
large enough. This implies that the solution $\{u^{n,R,\epsilon}\}$ is
at least $C^{0,\beta}$ by the arguments of the regularity proof above.
\smallskip

\noindent 2. The above argument also shows that, for fixed $\epsilon$,
the $C^{0,\beta}$-bounds for the $\{u^{n,R,\epsilon}\}$ are uniform in
$n$ since they depend only on $R$ through the $T_R$-term. This
allows us to pass to the limit $n\to + \infty$ and get a solution
$u^{R,\epsilon}:=\lim_{n\to+\infty} u^{n,R,\epsilon}$ of the limiting equation enjoying the
same $C^{0,\beta}$-bound.  This solution satisfies the truncated
viscous equation with $\mu_n$ replaced by the singular measure $\mu$.
\smallskip

\noindent 3. Next, we repeat the proof of the $C^{0,\beta}$-bound for
the truncated
viscous equation~:
 Estimate \eqref{est.Iux.Iuy}  together with the fact that $T_R$ is an increasing and a $1$-Lipschitz continuous function, implies that
$$T_R(I[u](x))-T_R(I[u](y))\leq k_2\,.$$
at least for $C$ big enough.
Rewriting the analogue of \eqref{keyineqcob} as
$$
-\epsilon \left[ u_{xx}(x)-u_{xx}(y)\right]  \leq \left[T_R(I[u])(x) - T_R(I[u])(y)\right] +2(||u||_\infty +
||f||_\infty)\; ,$$
this new estimates on the difference
of the truncated terms shows that the $C^{0,\beta}$-bound which is obtained in Step ~1,
is independent of $R$ and we can let $R\ra+\infty$. The result is
that the limit $u^\eps:=\lim_{R\to\infty}u^{R,\eps}$ is a
$C^{0,\beta}$-solution of the non-truncated viscous equation
$$-I[u] -\epsilon u_{xx} + u = f \quad \hbox{in  }\dom\; .$$

\noindent 4. Finally we come back again to the proof of the
$C^{0,\beta}$-bound but, this time, the main role is played by the
non-local term via estimate~\eqref{est.Iux.Iuy}.  Indeed we rewrite the analogue of \eqref{keyineqcob} as
$$
- \left[I[u](x) - I[u](y)\right]  \leq \epsilon \left[ u_{xx}(x)-u_{xx}(y)\right]+2(||u||_\infty +
||f||_\infty)\; ,$$
and remark that, since the
$u_{xx}$-terms satisfy \eqref{est.uxx.uyy}, the $\epsilon$-term in
\eqref{keyineqcob} can be estimated by $\epsilon k'_2$. Using \eqref{est.Iux.Iuy},  we obtain
again a contradiction for large enough $C$. The
argument is the same as in Step~3 with the roles of the local and
nonlocal terms exchanged. This also explains the terminology
``Sirtaki's method'', since Sirtaki is a danse where we exchange the
roles of the two feet as we exchange here the role of the $\epsilon u_{xx}$ and $I[u]$ terms. To conclude the argument, we have found that
the $C^{0,\beta}$-bound is independent of $\epsilon$, and we pass
to the limit as $\epsilon\to0$. We get a solution $u$ of the original
problem  belonging to $C^{0,\beta}$. Since this solution is unique, it is
the solution we are looking for.
\end{proof}

Now we turn to the case when $N\geq 2$. Unfortunately we require far
more retrictive assumptions on $f$.

\begin{theorem}\label{exist2}  Assume $N\geq2$, that
  $(\mathrm{H}_f)$, $(\mathrm{H}_\eta^6)$, and $(\mathrm{H}_\mu'')$
  hold, and that  $f(\dots,x_N)$ is in
  $W^{2,\infty}(\R^{N-1})$ for any $x_N>0$ with uniformly bounded
  $W^{2,\infty}$-norms.

\noindent (a) Any bounded, uniformly continuous
  solution of \eqref{E} 
is in  $C^{0,\beta}(\domb)$ for some $\beta >\tibeta$.

\noindent (b) There exists a solution of \eqref{E} in
$C^{0,\beta}(\domb)$ for some $\beta >\tibeta$.
\end{theorem}


\begin{proof} We are not going to provide the full proof since it is
  rather long and tedious and is mostly based on two ingredients which
  we have already seen.
  But we remark that an easy consequence of the the comparison result
  and linearity of the problem, is that $u$ inherits the regularity of
  $f$. I.e.  there exists a constant $K>0$ such that, for any $x',z'
  \in \R^{N-1}$ and $x_N>0$,
\begin{equation}\label{est.u.zprime}
 -K|z'|^2 \leq u(x'+z',x_N)+ u(x'-z',x_N) -2 u(x'+z',x_N) \leq K|z'|^2 \; .
\end{equation}

Then we repeat the $1$-d proof essentially considering
$$ \sup_{[0,+\infty)\times [0,+\infty)}\, (u(x',x_N)-u(x',y_N) - C|x_N-y_N|^\beta)\; .$$
Of course, a doubling of variables in $x'$ is necessary to take care
of the singularity of the measure, but using the $W^{2,\infty}$
property in $x'$, we can go back to the $1$-d computations without any
difficulty. Let us just mention the key decomposition we
use here. We rewrite the integrals with respect to $\mu_*$, first replacing
the integrands by
$$u(x'+z',x_N+z_N)+ u(x'-z',x_N-z_N)-2u(x',x_N),$$
and then by
$$\Delta^2_{z'}u(x',x_N+z_N)+\Delta^2_{z'}u(x',x_N-z_N)+2\Delta^2_{z_N}u(x',x_N),$$
where
\begin{align*}
\Delta^2_{z'}u(x',x_N):= & \frac{1}{2}\Big(u(x'+z',x_N)+u(x'+z',x_N)-2u(x',x_N)\Big)\,,\\
\Delta^2_{z_N}u(x',x_N):= & \frac{1}{2}\Big(u(x',x_N+z_N)+u(x',x_N-z_N)-2u(x',x_N)\Big)\,.
\end{align*}
These expressions are not equal pointwise of course,
but they give the same integrals because of the symmetry of $\mu_*$.
We deal with the $\Delta^2_{z'}$-terms using \eqref{est.u.zprime},
and the $\Delta^2_{z_N}$-term is treated as in the one dimensional
case. Also note that we use a decomposition of $\dom$ into
sets like $\R^{N-1}\times \{z_N:\ a \leq z_N \leq b\}$, for $a,b >0$,
following the $1$-d proof.

Finally, concerning the nonsymetric part $\mu_\#$, we use as usual the fact that
it is a controlled term since it is less singular.

The existence is proved as in the proof of Theorem~\ref{exist1}.
\end{proof}

\begin{remark}
The regularity results of the $N=1$ and $N\geq2$ cases are different.
In the first case, the results is purely elliptic and we gain regularity.
In the second case, the result is elliptic in the $x_N$-direction
while in the other directions we just use a preservation of regularity
argument. It is an open problem to find an elliptic argument also
in the $x'$-directions.
\end{remark}

\section{The limit as $\alpha\to2^-$}
\label{sec:limit}

\newcommand{\Id}{\mathrm{Id}}

In this section we prove that all the Neumann
models we consider converge to the same local Neumann problem as
$\alpha\to2^-$, provided that the nonlocal operators include
the normalisation constant $(2-\alpha)$. To be more precise, we consider the
following problem
\begin{align}
\label{alp-prob}
\begin{cases}
	-(2-\alpha)\int_{\R^N} u_\alpha(x+\eta(x,z))-u_\alpha(x)\,d\mu_\alpha + u_\alpha(x) = f(x)& \text{in } \Omega\,,\\[0.2cm]
		\dfrac{\partial u}{\partial \mathbf{n}}=0 & \text{in } \partial\Omega\,,
	\end{cases}
\end{align}
where $\alp\in(0,2)$, $\eta$ depends on the Neumann model we consider, and
$$\frac{d\mu_\alpha}{dz}=\frac{g(z)}{|z|^{N+\alpha}}$$
where $g$ is nonnegative, continuous and bounded in $\R^N$, $g(0)>0$ and $g\in C^1(B)$ for some ball $B$ around $0$.

We prove below that the solution of \eqref{alp-prob} converge to the
solution of the following local problem,
\begin{equation}\label{eq:local}\begin{cases}
		-a\Delta u - b\cdot Du +u = f & \text{in } \Omega\,,\\
		\qquad\dfrac{\partial u}{\partial \mathbf{n}}=0 & \text{in } \partial\Omega\,,
	\end{cases}\end{equation}
where
$$a:=g(0)\frac{|S^{N-1}|}{N}\quad\text{and}\quad b:=Dg(0)\frac{|S^{N-1}|}{N}.$$
In this section $|S^{N-1}|$ denotes the measure of the unit sphere in $\R^N$ and
$\Id_N$ the $N\times N$ identity matrix.

\begin{theorem}\label{thm:localisation}
  Assume $(\mathrm{H}_\eta^i),\ i=0\dots 4$ hold and let $u_\alpha$  be the solutions of \eqref{alp-prob} for $\alp\in(0,2)$. Then,
   as $\alpha\to2^-$,
  $u_\alpha$ converges locally uniformly to the unique
  solution $u$ of \eqref{eq:local}.
\end{theorem}

Before providing the proof, we introduce the following sequences of measures:
\begin{align*}
	(d\nu^1_\alpha)_{i,j}&=(2-\alpha)z_iz_j\frac{g(z)}{|z|^{N+\alpha}}\,dz\,,\\
	d\nu^2_\alpha&=(2-\alpha)z\frac{g(z)-g(0)}{|z|^{N+\alpha}}dz\,,\\
	(d\nu^3_{\alpha,y})_{i,j}&=(2-\alpha)\eta(y,z)_{i}\eta(y,z)_{j}\frac{g(z)}{|z|^{N+\alpha}}\,dz\,,\\
	d\nu^4_{\alpha,y}&=(2-\alpha)\eta(y,z)\frac{g(z)-g(0)}{|z|^{N+\alpha}}dz\,,
\end{align*}
where $\eta(y,z)_i$ denotes the $i$-th component of the vector
$\eta(y,z)$. Note that $\nu^1_\alpha$ and $\nu^3_{\alpha,y}$ are
matrix measures while $\nu^2_{\alpha}$ and $\nu^4_{\alpha,y}$ are
vector measures. The localization phenomenon occuring as
$\alpha\to2$ is reflected in the following lemma:
\begin{lemma}\label{lem:concentration}\

\noindent (a) As $\alpha\to 2^-$,
	$\nu^1_\alpha\rightharpoonup a\delta_0\Id_N$ and
$\nu^2_\alpha\rightharpoonup b\delta_0$
in the sense of measures.

\noindent (b)	For any sequence $\alpha_k\to2$ and $y_k\to x$, there
exist two vector functions $\bar{a}(x),\bar{b}(x)\in\R^N$  satisfying
	$$ \frac{1}{2}a\leq \bar{a}_i(x)\leq
        \Lambda\quad\text{and}\quad
        |\bar{b}_i(x)|\leq\Lambda \quad\text{for some $\Lambda
=\Lambda(g,\eta)<\infty$}\,,$$
        such that, at least along a subsequence,
	$$\nu^3_{\alpha_k,y_k}\rightharpoonup \mathop{\rm diag}(\bar{a}(x))\delta_0\,,
	\quad \nu^4_{\alpha_k,y_k}\rightharpoonup
        \bar{b}(x)\delta_0\,,$$ where $\mathop{\rm diag}(\bar{a}(x))$
        is the diagonal matrix with diagonal coefficients $\bar{a}_i(x)$.
\end{lemma}
\begin{proof}
	If $\delta\in(0,1)$ is fixed, we notice first that, for any $K>1$,
	$$0\leq(2-\alpha)\int_{\delta<|z|<K}|z|^2\frac{g(z)\,dz}{|z|^{N+\alpha}}\leq
	\|g\|_\infty(\delta^{2-\alpha}-K^{2-\alpha})\to0\text{ as
        }\alpha\to2^-\,,$$ so that the only possible limit in the
        sense of measure is supported in $\{0\}$.  Similar calculations
        show that the same is true for all the measures $\nu^i$,
        $i=2\dots 4$.

	Coming back to $\nu^1$, we  compute the inner integral as follows,
	\begin{align*}
		&(2-\alpha)\int_{|z|<\delta}z_iz_j\frac{g(z)}{|z|^{N+\alpha}}\,dz\\
                &=
		g(0)(2-\alpha)\int_{|z|<\delta}z_iz_j\frac{dz}{|z|^{N+\alpha}}+(2-\alpha)\int_{|z|<\delta}z_iz_j\frac{g(z)-g(0)}{|z|^{N+\alpha}}\,dz\,.
	\end{align*}
	The second integral vanishes as $\alpha\to2$ since
	$$\begin{aligned}
		&\Big|(2-\alpha)\int_{|z|<\delta}z_iz_j\frac{g(z)-g(0)}{|z|^{N+\alpha}}\,dz\Big|\\
&\leq
		C_g(2-\alpha)\int_{|z|<\delta}\frac{|z|^3}{|z|^{N+\alpha}}\,dz
		\leq C_g(2-\alpha)\frac{\delta^{3-\alpha}}{3-\alpha} \to 0 \quad \text{as  }\alpha \to 2^-,
	\end{aligned}$$
	for $C_g=\|Dg\|_{L^\infty(B_\delta)}$.
	By symmetry, the first integral is zero for $i\neq j$, while for $i=j$,
	$$\begin{aligned}
	g(0)(2-\alpha)\int_{|z|<\delta}z_i^2\frac{dz}{|z|^{N+\alpha}}&=
	g(0)\frac{|S^{N-1}|}{N}(2-\alpha)\int_{r=0}^\delta \frac{r^{2+N-1}}{r^{N+\alpha}}\,dr\\
	&=g(0)\frac{|S^{N-1}|}{N}\delta^{2-\alpha}\longrightarrow a\quad\text{as }\alpha\to2^-\,.
      \end{aligned}$$ This means that the measures $\{\nu^1_\alpha\}$
      concentrate to a delta mass $\delta_0$ multiplied by the
      diagonal matrix $a\Id_N$.

	Let us now consider the inner integral for each component
 of the measures $\nu^2_\alpha$ : using similar arguments, we have
	$$\begin{aligned}
	&(2-\alpha)\int_{|z|<\delta}z_i\frac{g(z)-g(0)}{|z|^{N+\alpha}}\,dz\\
        &=
	(2-\alpha)\int_{|z|<\delta}z_i\frac{(z,Dg(0))+(z,Dg(z)-Dg(0))}{|z|^{N+\alpha}}\,dz\\
	&=\sum_{j=1}^N\frac{\partial g}{\partial x_j}(0)(2-\alpha)\int_{|z|<\delta}z_iz_j
	\frac{dz}{|z|^{N+\alpha}}+o_\delta(1)\\
	&=\frac{\partial g}{\partial
          x_i}(0)(2-\alpha)\int_{|z|<\delta}z_i^2\frac{dz}{|z|^{N+\alpha}}+o_\delta(1)\\
        &\longrightarrow \frac{\partial g}{\partial x_i}(0)\frac{|S^{N-1}|}{N}+o_\delta(1)\quad\text{as}\quad\alpha\to2^-\,.
	\end{aligned}$$
	Hence, by the definition of $b$, $\nu^2_\alpha$ concentrates to $b\delta_0$.

	We now come to the measures $\nu^3$ which is more complex to
        analyse due to the presence of the perturbation
        $\eta(y_k,z)$. We first notice that by using
        $(\mathrm{H}_\eta^2)$, it follows that for $i\neq j$,
	$$\int_{|z|<\delta}\eta(y_k,z)_i\eta(y_k,z)_j\frac{g(z)dz}{|z|^{N+\alpha}}=0\,.$$
	Then by $(\mathrm{H}_\eta^1)$ $|\eta(y_k,z)|\leq c_\eta|z|$,
        and we have
	$$\begin{aligned}
	0&\leq g(0)(2-\alpha)\int_{|z|<\delta}\eta(y_k,z)_i^2
        \frac{dz}{|z|^{N+\alpha}}\,dz\\
        &\leq
	g(0)c_\eta^2(2-\alpha)\int_{|z|<\delta}|z|^2\frac{dz}{|z|^{N+\alpha}}\,dz\leq g(0)c_\eta^2|S^{N-1}|\,.
    \end{aligned}$$
    So, the total mass of $\nu^3$ is bounded and, by the same arguments as above, it is clear that the support of $\nu^3$ shrinks
	to $\{0\}$ (or the empty set).

    Then, we split the integral over $\{|z|<\delta\}$ as follows
    $$ (2-\alpha)\int_{|z|<\delta}\eta(y_k,z)_i^2\frac{g(z)dz}{|z|^{N+\alpha}}=
    \int_{|z|<\delta \atop z_N>-y_{k,N}}(\cdots)
    + \int_{|z|<\delta \atop z_N\leq -y_{k,N}}(\cdots) = (A_i)+(B_i)\,.$$
    The first integral is easy to handle since $\eta_i(y_k,z)=z$ when
    $z_N>-y_{k,N}$,
    $$\begin{aligned} (A_i) & = (2-\alpha)\int_{|z|<\delta \atop z_N>-y_{k,N}}z_i^2\frac{g(z)dz}{|z|^{N+\alpha}}\\
      & = (2-\alpha)\int_{|z|<\delta \atop
        z_N>0}z_i^2\frac{g(z)dz}{|z|^{N+\alpha}} + o(y_{k,N})
      \to\frac{1}{2}a\,.
    \end{aligned}$$
The other integral has a sign and can take
    different values according to the structure of the jumps, but in
    all cases we see that the weak limit of $\nu^3$ can be written as
    $\bar{a}(x)\delta_0$ where $\bar{a}(x)$ satisfies
    $a/2\leq\bar{a}_i(x)\leq\Lambda$.

    The measure $\nu^4$ is treated similarly:
	the total mass can be bounded by
	\begin{align*}
&	(2-\alpha)\int_{|z|<\delta}|\eta(y,z)|\frac{|g(z)-g(0)|}{|z|^{N+\alpha}}dz\\
&\leq
	c_\eta C_g (2-\alpha)\int_{|z|<\delta}|z|^2\frac{dz}{|z|^{N+\alpha}}=
	c_\eta C_g|S^{N-1}|\delta^{2-\alpha}\,,
	\end{align*}
	so that, up to a subsequence, there exists indeed a vector function $\bar b$
	such that $\nu^4_{\alpha_n,y_n}\to \bar b\delta_0$ in the sense of
	measures, with $\|\bar b\|_\infty\leq c_\eta C_g|S^{N-1}|$. The
	result then holds with $\Lambda:=|S^{N-1}|c_\eta\max\{C_g,g(0)\}.$
\end{proof}

\begin{remark} Note that in the censored case, $\bar{a}(x)\equiv a/2$ since
  the jumps below level $-y_N$ are censored, while
  $\bar{a}(x)=a$ by symmetry when the jumps are mirror reflected.
  Under our general hypotheses, different structures of the jumps
  (i.e. different  $\eta$'s) lead to different
  $\bar{a}$'s which could in principle depend on $x$ and the sequences
  $\alpha_k,y_k$.  We will overcome this difficulty by using the
  extremal Pucci operator associated to $\bar{a}(x)$: for any
  symmetric $N\times N$ matrix $A$ with eigenvalues $(\lambda_i)$ we
  define
\begin{equation}\label{abar.pucci}
    \mathcal{M}^+(A):=\frac{a}{2}\sum_{\lambda_i<0}\lambda_i + \Lambda\sum_{\lambda_i>0}\lambda_i\,.
\end{equation}
\end{remark}

\begin{proposition}\label{hrlalpha}
	Let us define the half relaxed limits as $\alpha\to2^-$,
	$$\bar{u}(x):=\limsup_{\alpha\to2, y\to x}u_\alpha(y)\qquad\text{and}\qquad\underline{u}(x):=\liminf_{\alpha\to2, y\to x}u_\alpha(y)\,.$$
	Under the assumptions of Theorem \ref{thm:localisation},
        $\bar{u}$ is a viscosity subsolution of \eqref{eq:local}, and
        $\underline{u}$ is a viscosity supersolution of
        \eqref{eq:local}.
\end{proposition}

\begin{proof}
  The proofs for $\bar{u}$ and $\underline{u}$ are similar, therefore we only
  provide it for $\bar{u}$.  We have to check that $\bar{u}$ satisfy
  the viscosity subsolution inequalities for the Neumann problem
  \eqref{eq:local} at any point $x\in\domb$. There are two separate
  cases to check, (i) when $x\in\dom$ and (ii) when $x\in\del\dom$.
\medskip

	\noindent\textsc{Step 1.} Case (i) where $x\in\dom$, that is
	$x_N>0$. Let $\phi$ be a smooth function and assume that $x$ is a strict local maximum point of $u-\phi$. By
        standard arguments there exists a sequence $(y_\alpha)_\alp$
        of local maximum points of $u_\alpha-\phi$ such that $y_\alpha\to x$
        as $\alpha\to2^-$.  Moreover, since $x_N>0$, by taking
        $\alpha$ close to $2$, we can assume
        that $y_{\alpha,N}>\delta$ for some small $\delta>0$. By the subsolution inequality for $u_\alpha$ at
        $y_\alpha$,
	$$\begin{aligned}
	 -(2-\alpha)\int_{|z|<\delta} \phi(y_\alp+z)-\phi(y_\alp)-D\phi(y_\alp)\cdot z \,d\mu_\alpha - (2-\alpha)\int_{|z|<\delta} D\phi(y_\alp)\cdot z \,d\mu_\alpha\\
	 -(2-\alpha)\int_{|z|\geq \delta}u_\alpha(P(y_\alp,z))-u_\alpha(y_\alp)\,d\mu_\alpha + u_\alpha(y_\alp) \leq f(y_\alp).
	\end{aligned}$$
We recall that the second integral of the left-hand side is well-defined : see the remark after Lemma~\ref{cor:compensator}.

We denote the three integral terms by $I_1$, $I_2$, and $I_3$. Then
	$$\begin{aligned}
	 I_1 
	 &=
         -(2-\alpha)\int_{|z|<\delta}((D^2\phi(y_\alp)+o_\delta(1))z,z)\,d\mu_\alpha
         \\ 
	 &=-(2-\alpha)\int_{|z|<\delta}(D^2\phi(y_\alp)z,z)\,d\mu_\alpha+o_\delta(1).
       \end{aligned}$$
       Note that the $o_\delta(1)$-term is independent of
       $\alpha$ because the measure $(2-\alpha)|z|^2\mu_\alpha$ has
       bounded mass.  The symmetry of $\mu_\alpha$ implies that
       $\int_{|z|<\delta}z_iz_j\,d\mu_\alpha=0$ and then, by Lemma
       \ref{lem:concentration}, we get
\begin{align*}
I_1&=-(2-\alpha)\mathrm{Tr}(D^2\phi(y_\alp))\int_{|z|<\delta}|z|^2\,d\mu_\alpha+o_\delta(1)
=
	 a\Delta\phi(x)+o_\alpha(1)+o_\delta(1)\,.
\end{align*}
	Similarly we have
	$$I_2=- (2-\alpha)\int_{|z|<\delta} D\phi(y_\alp)\cdot z \,d\mu_\alpha
	= -(2-\alpha)D\phi(y_\alp)\int_{|z|<\delta}z\,d\mu_\alpha,$$
and by symmetry of $\mu_\alpha$ and Lemma \ref{lem:concentration} we see that
	\begin{align*}
I_2 &= -(2-\alpha)D\phi(y_\alp)\int_{|z|<\delta}z\frac{g(z)-g(0)}{|z|^{N+\alpha}}\,dz
=b\cdot D\phi(x)+o_\alpha(1)+o_\delta(1)\,.
\end{align*}
The $o_\delta(1)$-terms are independent of $\alpha$ since the
measures $\nu^2_\alpha$ of Lemma \ref{lem:concentration} have unformly
bounded mass. For the last integral $I_3$, we use the boundedness of
 $(u_\alpha)_\alp$ with respect to $\alpha$ to see that
	\begin{equation}\label{bound.I3}
	|I_3|
        \leq
	C(2-\alpha)\int_{|z|\geq\delta}\frac{dz}{|z|^{N+\alpha}}\leq C'\frac{2-\alpha}{\alp\delta^\alpha}\ra0\quad\text{as}\quad\alp\ra2\,.
	\end{equation}
	So we keep $\delta>0$ fixed and pass to the limit as
$\alpha\to2^-$ (and $y_\alp\to x$) to get
	$$-a\Delta\phi(x)-b\cdot D\phi(x)+\bar{u}(x)\leq f(x)+o_\delta(1)\,.$$
	Then, since $\delta<x_N$ could be arbitrarily small, we pass to
        the limit as $\delta\ra0$ and get the viscosity subsolution
        condition for $\bar{u}$ at $x$.
\medskip

	\noindent\textsc{Step 2.} Case (ii) where $x\in\del\dom$, that is
	$x_N=0$. We again consider a smooth function $\phi$
	such that $\bar{u}-\phi$ has a strict local maximum point at $x$ and, as
above, we have a sequence $(y_\alpha)_\alp$
        of maximum points of $u_\alpha-\phi$ such that $y_\alpha\to x$
        as $\alpha\to2^-$.

        In this step we are going to prove that
	\begin{equation}\label{eq:boundary.local}
	\min\Big(-\mathcal{M}^+(D^2u(x))-\Lambda|Du(x)|+\bar{u}(x)-f(x)\
	;\ \frac{\partial\phi}{\partial\mathbf{n}}(x)\,\Big)\leq0\,,
	\end{equation}
	where $\mathcal{M}^+$ is defined in \eqref{abar.pucci}.
        We may assume that
        $\frac{\partial\phi}{\partial\mathbf{n}}(x)=-\frac{\partial\phi}{\partial{x_N}}(x)>0$ since otherwise
        \eqref{eq:boundary.local} is already satisfied.  Then for $\alpha$
        close to $2$, $-\frac{\partial\phi}{\partial{x_N}}(y_\alp)>0$
        by the continuity of $D\phi$. We can also
assume  $y_\alp\in\dom$, since otherwise $y_\alp\in\del\dom$ and
 then
 $\frac{\partial\phi}{\partial\mathbf{n}}(y_\alp)=-\frac{\partial\phi}{\partial{x_N}}(y_\alp)\leq0$
 for
$\alp$ close to 2 by Definition \ref{def1}, and this would contradict our assumption.

	Therefore $0<y_{\alpha,N}\to0$ as $\alpha\to2$, and
        the subsolution inequality for $u_\alp$ takes the form
	$$\begin{aligned}
	& -(2-\alpha)\int_{|z|<\delta} \phi(y_\alp+\eta(y_\alp,z))
	-\phi(y_\alpha)-D\phi(y_\alp)\cdot \eta(y_\alp,z) \,d\mu_\alpha\\
	& - (2-\alpha)\int_{|z|<\delta} D\phi(y_\alp)\cdot \eta(y_\alp,z) \,d\mu_\alpha\\
	& -(2-\alpha)\int_{|z|\geq \delta}u_\alpha(P(y_\alp)-u_\alpha(y_\alp)\,d\mu_\alpha + u_\alpha(y_\alp)
	\leq f(y_\alp).
	\end{aligned}$$
	We denote as before  the three integral terms by $I_1,I_2,I_3$.
	The compensator term $I_2$ can be written as
	$$\begin{aligned}
  	  I_2&= -g(0)D\phi(y_\alp)\cdot(2-\alpha)\int_{|z|<\delta}\eta(y_\alp,z)\frac{dz}{|z|^{N+\alpha}}\\
	  &\quad\, -D\phi(y_\alp)\cdot(2-\alpha)\int_{|z|<\delta}\eta(y_\alp,z)\frac{g(z)-g(0)}{|z|^{N+\alpha}}dz\\
		&=I_{2,1}+I_{2,2}\,.
              \end{aligned}$$
For symmetry reasons of both $\eta$ and
              the measure, $I_{2,1}$ reduces to the scalar product of
              the $N$-th components, and it has a sign,
	$$\begin{aligned}
  	  I_{2,1}&= -g(0)\frac{\partial\phi}{\partial{x_N}}(y_\alp)(2-\alpha)
  	  \int_{|z|<\delta}\eta(y_\alp,z)_N\frac{dz}{|z|^{N+\alpha}}\geq0\,\\
	\end{aligned},$$
since $g(0)$,
        $-\frac{\partial\phi}{\partial{x_N}}(y_\alp)$,
        and the $\eta$-integral are nonnegative (see Lemma
        \ref{lem:mu.sharp} (iii)).  Thus we may
drop the $I_{2,1}$ term from the inequality above and get that
	$$I_1+I_{2,2}+I_3+u_\alpha(y_\alp)\leq f(y_\alp)\,.$$

	We now pass to the limit in this inequality as $\alpha\to2$
        and hence $y_{\alpha,N}\to0$.  The difference with
        \textsc{Step 1} above, is that now $y_\alpha$ converge to the
        boundary so that we cannot take a fixed $0<\delta<y_{\alpha,N}$
as $\alp\ra2$.
For the first integral, Lemma~\ref{lem:concentration} enables us to take subsequences $\alpha_k\to2$ and $y_\alpha\to0$ such that
(dropping the subscript $k$ for simplicity)
	$$\begin{aligned}
	I_1&= -(2-\alpha)\int_{|z|<\delta} \phi(y_\alp+\eta(y_\alp,z))-\phi(y_\alp)-D\phi(y_\alp)\cdot
	\eta(y_\alp,z)\frac{g(z)dz}{|z|^{N+\alpha}}\\
	&=-\sum_{i,j}\int_{|z|<\delta}(\partial^2_{i,j}\phi(y_\alp) + o_\delta(1))\,d(\nu^3_{\alpha,y_\alp})_{i,j}(z)
	\\
	&=-\sum_{i,j}\partial^2_{i,j}\phi(x)\int_{|z|<\delta}d(\nu^3_{\alpha,y_\alp})_{i,j}(z)+o_{\alp}(1)+o_\delta(1)\\
	&= -\sum_i\bar{a}_i(x)\partial_{i,i}^2\phi(x) + o_{\alpha}(1)+o_\delta(1)\\
	& \geq -\mathcal{M}^+(D^2\phi(x))+ o_{\alpha}(1)+o_\delta(1)\,.
	\end{aligned}$$

    The last term $I_3$ can be treated as in \textsc{Step 1} and
    vanishes as $\alpha\to2$. We are left with the $I_{2,2}$ term
and use again Lemma \ref{lem:concentration}, this time for the
    measure $\nu^4$. The result is the existence of a vector
$\bar{b}(x)$ such that along subsequences we have
	$$\begin{aligned}
	I_{2,2}&=D\phi(y_\alp)\cdot(2-\alpha)\int_{|z|<\delta}\eta(y_\alp,z)\frac{g(z)-g(0)}{|z|^{N+\alpha}}dz\\&=
    D\phi(x)\cdot \bar{b}(x)+o_{\alpha}(1)
    \geq -\Lambda|D\phi(x)|+o_{\alpha}(1)\,.
	\end{aligned}$$
Hence, passing to the limit  $\alpha\to2$ in the above inequality, leads to
$$-\mathcal{M}^+(D^2\phi(x))-\Lambda|D\phi(x)|+\bar{u}(x)-f(x)\leq0, $$
and \eqref{eq:boundary.local} still holds.
\medskip

\noindent\textsc{Step 3.} We shall prove now that boundary condition
\eqref{eq:boundary.local} reduces to the condition
$\frac{\partial\phi}{\partial\mathbf{n}}\leq0$. Let us assume on the contrary
that $\frac{\partial\phi}{\partial\mathbf{n}}(x)>0$ for some point $x$ at
the boundary $\{x_N=0\}$ and some smooth function $\phi$  such that
$u-\phi$ has a maximum point at $x$. For any $\tau,\eps>0$, we
take a smooth, bounded function $\psi:\R_+\to\R_+$ such that
	$$\psi(t)=\tau\Big(t - \dfrac{t^2}{\eps^2}\Big)\quad\text{for}\quad 0\leq t\leq\eps^2/2\,.$$
    Since $\psi(0)=0$ and $0\leq\psi$ for $0\leq t\leq\eps^2/2$, it follows that
 $u(x)-\phi(x)-\psi(x_N)$ has again a local maximum point at $x$.
    Hence \eqref{eq:boundary.local} holds with $\phi(x)+\psi(x_N)$
    replacing $\phi(x)$, i.e.
	\begin{equation}\label{eq:local.boundary.2}
          \min\Big(E(\phi)+\frac{a}{2}\frac{2\tau}{\eps^2}-\Lambda\tau ;\ \frac{\partial\phi}{\partial\mathbf{n}}(x)-\tau\,\Big)\leq0\,,
	\end{equation}
where
    $$E(\phi):=-\mathcal{M}^+(D^2\phi(x))-\Lambda|D\phi(x)|+\bar{u}(x)-f(x)\,.$$
        Since we assumed that $\frac{\partial\phi}{\partial\mathbf{n}}(x)>0$, we
        first fix $\tau>0$ small enough so that the inequality
        $\frac{\partial\phi}{\partial\mathbf{n}}(x)-\tau>0$ still holds.
    Then we can choose $\eps>0$ small enough to ensure that also
    $$E(\phi)+\frac{a}{2}\frac{2\tau}{\eps^2}-\Lambda\tau>0\,.$$
    But then we contradict \eqref{eq:local.boundary.2}, and  hence the
    boundary condition for $\bar{u}$ reduces to
    $\partial\phi/\partial\mathbf{n}\leq0$
    everywhere on the boundary. This concludes the proof of Proposition~\ref{hrlalpha}.
\end{proof}

\begin{proof}[Proof of Theorem \ref{thm:localisation}]
  We have seen that $\bar{u}$ is a subsolution of \eqref{eq:local}
  while $\underline{u}$ is a supersolution of the same problem.  Since
  $\underline{u}\leq\bar{u}$ on $\domb$ by definition and
  $\underline{u}\geq\bar{u}$ on $\domb$ by the comparison principle for
  \eqref{eq:local}, we see that $\underline{u}=\bar{u}$ on $\domb$. Setting $u:=\underline{u}=\bar{u}$ on $\domb$, it
  immediately follows that $u$ is a continuous (since $\underline{u}$ is lsc and $\bar{u}$ is usc) and the unique viscosity solution of \eqref{eq:local}. By classical arguments in the half-relaxed limit method, the
sequence $(u_\alp)_\alp$ also converge locally uniformly to $u$.
\end{proof}

\appendix
\section{Blowup supersolution in censored case I.}
\label{App:Blowup}

In this section we assume $(\mathrm{H}_\eta^6)$ and $(\mathrm{H}_\mu')$
 as in Section \ref{Sec:alp1}. Remeber that
$\dom:=\big\{(x_1,\dots,x_N)=(x',x_N): x_N\geq 0\big\}$. First we show
that in the censored fractional Laplace case (i.e. the censored alpha stable
case), we can essentially take
$$\U(x)=-\ln x_N$$
as our blowup supersolution in assumption (U) in Section \ref{Sec:alp1}.

\begin{lemma}
\label{alph1}
If $d\mu(z)=\frac {dz}{|z|^{N+\alp}}$ for $\alp\in(0,1)$ and
$\U(x)=-\ln(x_N)$, then
$$-I[\U](x)=-\int_{x_N+z_N\geq0} \U(x+z)-\U(x)\
\frac{dz}{|z|^{N+\alp}}>0\qquad\text{for}\qquad x\in\dom.$$
\end{lemma}
\begin{proof}
We first change
variables, $\bz=\frac z {x_N}$, to find that
\begin{align*}
-I[\U](x)=\int_{x_N+z_N\geq0}\ln\Big(1+\frac z x\Big)\
\frac{dz}{|z|^{N+\alp}}=\frac1{x_N^\alp}
\int_{\bz_N\geq-1}\ln(1+\bz)\ \frac{d\bz}{|\bz|^{N+\alp}}.
\end{align*}
Now we are done if we can prove that
$$J=\int_{\bz_N\geq-1}\ln(1+\bz)\
\frac{d\bz}{|\bz|^{N+\alp}}>0.$$

When $N=1$, we take $1+\bar z=e^y$ and note that simple
computations lead to
\begin{align*}
J=\int_{-\infty}^\infty y \frac{e^{y}dy}{|e^y-1|^{1+\alp}}
=\int_{-\infty}^\infty F(y)e^{\frac y2(1-\alp)}dy \quad \text{where}\quad F(y)=\frac{y}{|2\sinh\frac y2|^{1+\alp}}.
\end{align*}
Since $F(y)$ is odd and $1-\alp>0$,
$$0<-F(-y)e^{-\frac  y2(1-\alp)}< F(y)e^{\frac
  y2(1-\alp)}\qquad\text{for}\qquad y>0,$$
and hence by symmetry $J>0$.

In the case $N>1$ we introduce polar coordinates $z=ry$ where $r\geq0$
and $|y|=1$, and we let $dS(y)$ be the surface measure of the sphere
$|y|=1$ in $\R^N$. We then find that
\begin{align*}
J=\Big(\int_{|y|=1,y_N>0}\int_0^\infty+\int_{|y|=1,y_N<0}\int_0^{-\frac1{y_N}} \Big) \ln(1+ry_N) \frac{r^{N-1}dr\,dS(y)}{r^{N+\alp}}.
\end{align*}
The change of variables $s=y_Nr$ then leads to
\begin{align*}
J&=\Big(\int_{|y|=1,y_N>0}\int_0^\infty+\int_{|y|=1,y_N<0}\int_0^{-1}
\Big) \sgn(y_N)|y_N|^\alp \ln(1+s) \frac{ds}{|s|^{1+\alp}}\,dS(y)\\
&=\int_{|y|=1,y_N>0}|y_N|^\alp\ dS(y)\ \int_{-1}^\infty\ln(1+s) \frac{ds}{|s|^{1+\alp}}.
\end{align*}
The lemma now follows from the computations for we did for  $N=1$.
\end{proof}

We now generalize to a much larger
class of integral operators with L\'evy measures $\mu$ such that
$d\mu(z)\sim\frac{dz}{|z|^{N+\alp}}$ near $|z|=0$. In this case the
blowup supersoution will be the modified log-function $\U_R$
defined as
$$\U_R(x)=\bar\U_R(x_N)\qquad\text{for}\qquad x\in\dom,\quad R>1,$$
where $\bar\U_R$ is a (nonnegative) monotone decreasing
$C^\infty(0,\infty)$ function such that
$$\bar\U_R(s)=\begin{cases}-\ln(s)+\frac32\ln R&\text{if } 0<s\leq R,\\
0 &\text{if } s\geq {2R}.
\end{cases} $$
The main result in this appendix says that $\U_R$ will
be the blowup ``supersolution'' of assumption (U) provided
the L\'evy measure $\mu$ also satisfies:\\[-0.3cm]

\begin{itemize}
\item[$(U)'$] For all $R,\eps>0$ there are $r,c,K>0$ and
  $\alp\in(0,1)$ such that
\begin{align*}
(a) &\int_{-1<z_N\leq R}\ln(1+z_N)
\Big(s^\alp\mu(sdz)-\frac{c\,dz}{|z|^{N+\alp}}\Big)>-\eps &&\text{for}\quad
s\in(0,r),\\
(b)&\int_{-1<z_N\leq -\frac 1{2}}\ln(1+z_N)\
\mu(sdz)\geq -K &&\text{for}\quad
s\in(r,R).
\end{align*}
\end{itemize}
\medskip

\begin{theorem}
\label{ThmB}
Assume $(\mathrm{H}_\eta^6)$, $(\mathrm{H}_\mu)'$, and $(U)'$
hold. Then the function $\U_R$ defined
above satisfy the assumptions in (U). In particular, there is
$R_0>0$ such that for any $R>R_0$ there is $K_R\geq0$ such that
\begin{align*}
-I[\U_R](x)&\geq-K_R\quad\text{in}\quad \{x:0<x_N\leq R\}.
\end{align*}
\end{theorem}

Before we prove this result, we show how assumption
$(U)'$ can be checked when $\mu$ is L\'evy measure whos
restriction to $\{z:|z_N|\leq r\}$ has a density
\begin{align}
\label{g_def}
\frac{d\mu}{dz}=\frac{g(z)}{|z|^{N+\alp}}\quad\text{where}\quad
\begin{cases}
\alp\in(0,1),\\[0.2cm]
0\leq g\in L^\infty_{\mathrm{loc}}(\R^N)\cap
L^1(\R^N;\frac{dz}{1+|z|^{N+\alp}}),\\[0.2cm]
\lim_{z\ra 0}g(z)=g(0)>0.
\end{cases}
\end{align}
Note that the $L^1$ assumption makes $\frac{d\mu}{dz}$ integrable near
infinity and that $L^\infty(\R^N)\subset L^1(\R^N;\frac{dz}{1+|z|^{N+\alp}})$
for $\alp>0$.

\begin{corollary}
\label{g_cor}
If $\mu$ has a density satisfying \eqref{g_def}, then the function
$\U_R$ defined above satisfy the assumptions in (U).
\end{corollary}
\begin{proof}
By Theorem \ref{ThmB} we have to check that $(U)'$ holds. Part (b)
follows from H\"older's inequality since $\ln(1+s)\in
L^1(-1,0)$. Now we check part (a). Note that
$$s^\alp\mu(sdz)-\frac{c\,dz}{|z|^{N+\alp}}=\frac{g(sz)-c}{|z|^{N+\alp}}dz.$$
Now choose $c=g(0)$ and write
\begin{align*}
&\int_{-1<z_N\leq R}\ln(1+z_N)
\Big(s^\alp\mu(sdz)-\frac{c\,dz}{|z|^{N+\alp}}\Big)\\
&\geq
-\sup_{-s<r<Rs}|g(r)-g(0)| \int_{-1<z_N\leq R}|\ln(1+z_N)|\frac{dz}{|z|^{N+\alp}}.
\end{align*}
Part (a) now follows since the last integral is finite for any $R>0$,
while the $\sup$-term goes to zero as $s\ra0$ by continuity of $g$ at $z=0$.
\end{proof}

\begin{remark}
\label{RemA1}
Assumption \eqref{g_def} also includes measures like
$$\mu=\sum_{i=1}^{M_1}\mu_i,$$
where  $\mu_i$ have densities satifying \eqref{g_def} for different
$g_i$ and $\alp_i$. To see this, simply take $\alp=\max_i\alp_i$ and
$g(z)=\sum_{i=1}^Mg_i(z)|z|^{\alp-\alp_i}$ and note that $g\in
L^1(\R^N;\frac{dz}{1+|z|^{N+\alp}})$. We can even relax this assumption
to include measures with zero or arbitrary negative $\alp_i$ provided
that $\max_i\alp_i$ remains in $(0,1)$. Finally we mention that we need
some assumption to insure that $\mu$ does not give to much mass to
the negative part of the integral $-I[\U_R]$. In \eqref{g_def} we do
this by requiring continuity at $0$ of $g$, but a carefull reader can
extend this assumption to allow some discontinuities at $0$.
\end{remark}

\begin{remark}
In assumption $(U)'$ it is only the restriction of $\mu$ to the set
$$\{z:-r<z_N<Rr\}\cap\{z:-1<z_N<-\frac12\}$$
that plays any role. Hence if $\mu$ satisfies $(U)'$, by taking
$\bar r$ small enough, so will
$\mu+\bar\mu$
for any measure $\bar\mu$ satisfying
$$\int_{|z|>0}d\bar\mu<\infty\qquad\text{and}\qquad
\supp \bar\mu \cap \{z:-1<z_N<\bar r\}=\emptyset\quad\text{for some
  $\bar r>0$}.
$$
E.g. the delta-measure $\bar\mu=\sum_{i=1}^{M}\delta_{x^i}$ is ok if $x^i_N>0$.
\end{remark}

\begin{proof}[Proof of Theorem \ref{ThmB}]
First note that there is an $R_0>0$ such that
\begin{align*}
J_{R_0}:=\int_{-1<z_N\leq R_0} \ln(1+z_N)
  \frac{dz}{|z|^{N+\alp}}>0.
\end{align*}
Indeed, in the proof of Lemma \ref{alph1}, we showed that
$J_\infty=J>0$. The result then follows by the Dominated Convergence
Theorem since the integrand is positive for $z_N>0$ and integrable.

For any $R>R_0$, we  note immediatly that $\U_R$ is a nonnegative
decreasing function which trivially satisfies the second part of (U)
with $\omega_R(s)=\frac1{\bar\U_R(s)}$.
We will now check that $\U_R$ has the appropriate supersolution
properties and hence complete the proof that $\U_R$ satisfies (U)
under $(U)'$ . By the definition of $\U_R$, we can write
\begin{align*}
-I[\U_R](x)&=\int_{-x_N<z_N\leq Rx_N} \ln\Big(1+\frac {z_N}{x_N}\Big)\
\mu(dz)+ I_R\\
&=\int_{-1<y_N\leq R} \ln(1+y_N)\ \mu(x_N\,dy)+I_R,
\end{align*}
where $I_R=-\int_{z_N>Rx_N} \U_R(x+z)-\U_R(x)\ \mu(dz)>0$ since
$\U_R$ is decreasing. By assumption $(U)'$  we then find a $r>0$ such
that for $x_N\in(0,r)$,
\begin{align*}
-x_N^{\alp}I[\U_R](x)&\geq J_R+\int_{-1<y_N\leq R} \ln(1+y)\
\Big(x_N^\alp\mu(x_N\,dy)-\frac{dy}{|y|^{1+\alp}}\Big)\geq \frac12J_R>0.
\end{align*}
When $x_N\in(r,R)$, another application of $(U)'$  along with
$(\mathrm{H}_\mu)'$ leads to
\begin{align*}
&-I[\U_R](x)\\
&\geq\Big(\int_{-x_N<z_N<-\frac {x_N}2}+\int_{-\frac{x_N}2<z_N<R\, \cap
  |z|<1}+\int_{-\frac{x_N}2<z_N<R\, \cap|z|>1}\Big)\ln\Big(1+\frac
z{x_N}\Big)  d\mu(dz)\\
&\geq -K - \underset{s\in(-\frac 12,\frac R r)}{\max}\frac{|\del_s\ln(1+s )|}{|x_N|} \int_{|z|<1}|z|
d\mu(z) -\underset{s\in(-\frac 12,\frac R r)}{\max}|\ln(1+s)|\int_{|z|>1}d\mu(z).
\end{align*}
Since this last expression is bounded for $x_N\in(r,R)$, this
completes the proof.
\end{proof}




\section{Estimates for the censored case II.}


\begin{lemma}
\label{alp-lem}
Let  $\mu(dz)=\frac{dz}{|z|^{N+\alp}}$, $\alp\in(1,2)$,  and define
$\tilde\theta(x)=|x_N|^\beta$. If $\beta\in(0,1)$ and $x\in\dom$, then
$$I[\tilde\theta](x)=P.V.\int_{x_N+z_N\geq0}\tilde\theta(x+z)-\tilde\theta(x)\ \mu(dz)\begin{cases}
>0 & \text{if}\quad \beta>\alp-1,\\
=0 & \text{if}\quad \beta=\alp-1,\\
<0 & \text{if}\quad \beta<\alp-1.
 \end{cases} $$
\end{lemma}
\begin{proof}
First let $\beta\in(0,1)$ and $N=1$, and define
$\tilde\theta(x)=|x|^\beta$.
Note that the change of variables $z=x\bz$
followed by $1+\bz=e^s$ reveals that
\begin{align*}
I[\tilde\theta](x)&=P.V.\int_{x+z\geq0}|x+z|^\beta-|x|^\beta \
\frac{dz}{|z|^{1+\alp}}\\
&=|x|^{\beta-\alp}\ P.V.\int_{\bar z\geq
  -1}|1+\bz|^\beta -1 \ \frac{d\bz}{|\bz|^{1+\alp}}\\
&=|x|^{\beta-\alp}\ P.V.\int_{-\infty}^\infty\frac{2\sinh\frac{\beta s}{2}}{|2\sinh
  \frac s2|^{1+\alp}}e^{\frac s2(1+\beta-\alp)}dx.
\end{align*}
When $\beta=\alp-1$, the integrand is odd and hence the integral is
zero. For $\beta>\alp-1$ ($\beta<\alp-1$) the exponential factor makes
the integral positive (negative). Hence when $\beta+1-\alp=0$, $>0$ or
$<0$, then $I[\tilde\theta]=0$, $>0$, or $<0$ respectively.

When $N>1$, a similar result holds for $\tilde\theta(x)=|x_N|^{\beta}$. The
idea is to work in polar coordinates. We set $x=ry$ for $r\geq0$ and
$|y|=1$ and let $dS(y)$ denote the surface area element of the $N$-sphere
$|y|=1$. We also use the change of variables $ry_N=\bar rx_N$.
\begin{align*}
&I[\tilde\theta](x)\\
&=\int_{x_N+z_N\geq0}|x_N+z_N|^\beta-|x_N|^\beta \
\frac{dz}{|z|^{N+\alp}}\\
&=\int_{|y|=1}\int_{x_N+ry_N>0}|x_N+ry_N|^\beta-|x_N|^\beta \
\frac{r^{N-1}dr\,dS(y)}{r^{N+\alp}}\\
&=\Bigg(\int_{|y|=1,y_N>0}\int_0^\infty+
\int_{|y|=1,y_N<0}\int_0^{-\frac{x_N}{y_N}}\Bigg)
(\cdots)\frac{dr\,dS(y)}{r^{1+\alp}} \\
&=\Bigg(\int_{|y|=1,y_N>0}\int_0^\infty-
\int_{|y|=1,y_N<0}\int_0^{-1}\Bigg)
|x_N|^{\beta-\alp}|y_N|^\alp\Big(|1+\bar r|^\beta -1\Big)
\frac{d\bar r\,dS(y)}{|\bar r|^{1+\alp}} \\
&=|x_N|^{\beta-\alp}\int_{|y|=1,y_N>0}|y_N|^{\alp}\ dS(y)\int_{-1}^\infty
|1+\bar r|^\beta -1\ \frac{d\bar r}{|\bar r|^{1+\alp}}.
\end{align*}
Here the first integral is just a positive constant while the second
integral is the same we found in the $N=1$ case. The conclusion is
therefore as in that case: When $\beta+1-\alp=0$, $>0$ or $<0$, then
$I[\tilde\theta]=0$, $>0$, or $<0$ respectively.

\end{proof}
Next we consider the two integrals
\begin{align*}
B(a)&=\int_{-a-1}^{-a}|1+z|^\beta-1 \ d\mu(z),\\
G&=\int_{-1}^1 |1+z|^\beta + |1-z|^\beta -2\ d\mu(z),
\end{align*}
where $a>1$, $\beta\in(0,1)$, and $d\mu(z)=\frac{dz}{|z|^{N+\alp}}$
for $\alp\in[1,2)$.
\begin{proposition}
\label{PropBG}
If $\beta=\alp-1$ then there is a $\kappa>0$ such that
$$B(a)+G\leq-\kappa<0$$
for any $a>1$.
\end{proposition}

By continuity of the integrals in $\beta$ we have the
following corollary:
\begin{corollary}\label{nec-res}
There is $\kappa>0$ and $\beta>\alp-1$ such that
$$B(a)+G\leq-\kappa\leq0$$
for any $a>1$.
\end{corollary}

To prove Proposition \ref{PropBG}, note that $z+1\leq0$ for
$z\in(-a-1,-a)$ ($a>1$) and that the
change of variable $1+z=-e^x$ in $B(a)$ leads to
$$B(a)=\int_{\ln(a-1)}^{\ln a}\frac{2\sinh\frac{\beta x}2}{|2\cosh
  \frac x2|^{1+\alp}}e^{\frac x2(1+\beta-\alp)}\
dx\overset{\beta=\alp-1}{=}\int_{\ln(a-1)}^{\ln
  a}\frac{2\sinh\frac{\beta x}2}{|2\cosh \frac x2|^{1+\alp}}\ dx.$$
For the $G$ integral we have the following result.
\begin{lemma}
\begin{gather*}
G=2\ P.V.\int_{-\ln 2}^{\ln 2}\frac{2\sinh\frac{\beta x}2}{|2\sinh
  \frac x2|^{1+\alp}}e^{\frac x2(1+\beta-\alp)} dx\ -\ 2\int_{\ln
  2}^{\infty}\frac{2\sinh\frac{\beta x}2}{|2\sinh
  \frac x2|^{1+\alp}}e^{-\frac x2(1+\beta-\alp)} dx,\\
\intertext{and if $\beta=\alp-1$,}
G=-\ 2\int_{\ln
  2}^{\infty}\frac{2\sinh\frac{\beta x}2}{|2\sinh
  \frac x2|^{1+\alp}}dx.
\end{gather*}
\end{lemma}
\begin{proof}
First note that by symmetry
$$G=2\lim_{b\ra0^+}\int_{(-1,1)\setminus(-b,b)} |1+z|^\beta -1\
d\mu(z).$$
Then, since $1+z>0$ for $z\in(-1,1)$, the change of variable
$1+z=e^x$ leads to
$$G=2\lim_{b\ra0^+}\int_{(-\infty,\ln 2)\setminus(\ln(1-b),\ln(1+b))}\frac{2\sinh\frac{\beta x}2}{|2\sinh
  \frac x2|^{1+\alp}}e^{\frac x2(1+\beta-\alp)}\ dx.$$
Note that $\ln(1\pm b)=\pm b+O(b^2)$ and decompose the above integral
as follows,
\begin{align*}
&\int_{(-\infty,\ln 2)\setminus(\ln(1-b),\ln(1+b))}(\cdots)\ dx\\
&= \Big(\int_{(-\infty,\ln 2)\setminus(-b,+b)}+\int_{(-\infty,\ln
  2)\setminus(\ln(1-b),-b)}-\int_{(-\infty,\ln
  2)\setminus(\ln(1+b),b)}\Big) (\cdots)\ dx
\end{align*}
Now since $\sinh x=x+O(x^3)$, the last two integrals are bounded by
$Cb^2 \frac b{b^{1+\alp}}=Cb^{2-\alp}$ for $b\ll1$, and we have
$$G=2\lim_{b\ra0^+}\int_{(-\infty,\ln 2)\setminus(-b,b)}(\cdots)\
dx=2\Bigg( P.V. \int_{(-\ln 2,\ln 2)}+\int_{(-\infty,-\ln 2)}\Bigg)(\cdots)\
dx.$$
A change of variables in the last integral then gives the first
statement of the Lemma. The last part of the lemma follows since the
integrand is odd when $\beta=\alp-1$, and hence the integral over
$(-\ln 2,\ln 2)$ vanishes.
\end{proof}
 We also need the next lemma.

\begin{lemma}
\label{LemBG}
If $\beta=\alp-1$, then $B(2)<-\frac G2$.
\end{lemma}

\begin{proof}
We will show that
$$B(2)=\int_0^{\ln 2}\frac{2\sinh\frac{\beta x}2}{|2\cosh
  \frac x2|^{1+\alp}}\ dx\leq \int_0^{\ln 2}\frac{2\sinh\frac{\beta (x+\ln2)}2}{|2\sinh
  \frac {x+\ln2}2|^{1+\alp}} \ dx<\int_0^\infty(\cdots)\ dx=-\frac G2.$$
The last inequality is trivial, and since $\sinh$ is an increasing
function, the first inequaliy follows if we can show that
$$\cosh \frac x2\geq \sinh \frac {x+\ln2}2\quad\text{for all}\quad
x\in(0,\ln2).$$
But this easily follows since $f(x)=\cosh \frac x2-\sinh \frac {x+\ln2}2$
satisfy
\begin{align*}
f'(x)&=\frac 12 \sinh  \frac x2 - \frac12\cosh \frac {x+\ln2}2\leq 0
\quad\text{for all $x$},\\
f(\ln 2)&
=\frac{\sqrt 2-1}4\geq 0.
\end{align*}
\end{proof}
\begin{proof}[Proof of Proposition \ref{PropBG}]
Divide the integral $B(a)$ into three parts
$$\Big(\int_{\ln(a-1)\wedge 0}^0+ \int_{\ln (a-1)\vee 0}^{\ln
  2\wedge\ln a}+\int_{\ln 2}^{\ln 2\vee\ln a}\Big)
(\cdots)\ dx.$$
Now we conclude since the first integral is negative, the second one
is less than $-\frac G2$ by Lemma \ref{LemBG}, and the last one is less
than $-\frac G2$ by definition of $G$.
\end{proof}


\end{document}